\documentclass[11pt]{amsart}

\usepackage[english]{babel}

\usepackage{amsmath}
\usepackage{graphicx}
\usepackage[colorlinks=true, allcolors=blue]{hyperref}
\usepackage{amsthm}
\usepackage{amssymb}
\usepackage{amsfonts,amscd}
\usepackage{fullpage}

\usepackage{tikz}
\usetikzlibrary{positioning, automata, arrows, shapes, matrix, arrows, decorations.pathmorphing, decorations.pathreplacing}
\tikzset{main node/.style={circle,draw,minimum size=0.3em,inner sep=0.5pt}}
\tikzset{small node/.style={circle, draw,minimum size=0.1cm,scale=0.3,fill}}
\tikzset{anchorbase/.style={>=To,baseline={([yshift=-0.5ex]current bounding box.center)}}}
\tikzset{->-/.style={decoration={markings, mark=at position 0.53 with {\arrow[-Latex]{>}}},postaction={decorate}}}
\usepackage{xspace}
\usepackage{adjustbox}
\theoremstyle{plain}
\newtheorem{theorem}{Theorem}[section]
\newtheorem{lemma}[theorem]{Lemma}
\newtheorem{prop}[theorem]{Proposition}

\newtheorem{cor}[theorem]{Corollary}
\theoremstyle{definition}
\newtheorem{defn}[theorem]{Definition}
\newtheorem{exa}[theorem]{Example}
\newtheorem{notation}[theorem]{Notation}
\theoremstyle{remark}
\newtheorem{rmk}[theorem]{Remark}
\numberwithin{equation}{section}
\newcommand{\al}{\alpha}
\newcommand{\be}{\beta}
\newcommand{\ep}{\varepsilon}
\newcommand{\Q}{{\mathbb Q}}
\newcommand{\R}{{\mathbb R}}
\newcommand{\Z}{{\mathbb Z}}
\newcommand{\ra}{\rightarrow}
\newcommand{\bfa}{\mathbf{a}}
\newcommand{\B}{{\mathcal B}}
\newcommand{\yg}[3]{Y_{#1,#2,#3}}
\newcommand{\ypqr}{\yg{a}{b}{c}}
\newcommand{\ygc}[3]{Y_{#1,#2,#3}}
\newcommand{\hexg}[3]{H_{#1,#2,#3}}
\newcommand{\hpqr}{\hexg{a-2}{b-2}{c-2}}

\newcommand{\Stab}{\text{\rm Stab}}
\newcommand{\tworoot}{2-root\xspace}
\newcommand{\tworoots}{2-roots\xspace}
\newcommand{\elementary}{elementary\xspace}
\newcommand{\simproots}{\Pi}
\newcommand{\posroots}{\Phi_+}
\newcommand{\allroots}{\Phi}
\newcommand{\realroots}{\Phi_{\text{\rm re}}}
\newcommand{\alltworoots}{\Phi^2}
\newcommand{\realtworoots}{\Phi^2_{\text{\rm re}}}
\newcommand{\otherb}{B'}
\newcommand{\thirdb}{\widetilde{B}}
\newcommand{\perm}{\text{\rm perm}}
\newcommand{\radi}{\text{\rm rad}}
\newcommand{\Span}{\text{\rm Span}}
\newcommand{\virasoro}{\omega}
\newcommand{\trorbit}{X}
\newcommand{\ome}{\rho}
\newcommand{\rootht}{\text{\rm ht}}
\newcommand{\tworootht}{\text{\rm ht}_2}

\makeatletter
\newcommand{\exterior}[1]{\mathop{\mathpalette\exterior@{#1}}}
\newcommand{\exterior@}[2]{%
  \raisebox{\depth}{%
  \fontsize{\sf@size}{0}%
  \m@th
  $\ifx#1\displaystyle\textstyle\else#1\fi\bigwedge$}%
  ^{\mspace{-2mu}#2}%
  \kern-\scriptspace
}
\makeatother

\makeatletter
\@namedef{subjclassname@2020}{\textup{2020} Mathematics Subject Classification}
\makeatother

\title{2-roots for simply laced Weyl groups}

\author{R. M. Green}
\address{Department of Mathematics\\
University of Colorado Boulder, Campus Box 395\\
Boulder, Colorado\\
USA, 80309}
\email{rmg@colorado.edu}

\author{Tianyuan Xu}
\address{Department of Mathematics and Statistics\\
    Haverford College\\
    Haverford, Pennsylvania\\
USA, 19041}
\email{txu2@haverford.edu}



\begin{document}
\maketitle

\begin{abstract}
We introduce and study ``\tworoots{}", which are symmetrized tensor products of orthogonal roots of
Kac--Moody algebras. We concentrate on the case where 
$W$ is the Weyl group of a simply laced Y-shaped Dynkin diagram $\ypqr$ having $n$ vertices and 
with three branches of arbitrary finite lengths $a$, $b$ and $c$; special cases of this include types $D_n$, 
$E_n$ (for arbitrary $n \geq 6$), and affine $E_6$, $E_7$ and $E_8$. We show that a natural
codimension-$1$ submodule $M$ of the symmetric square of the reflection representation of $W$ has a 
remarkable canonical basis $\B$ that consists of \tworoots.
We prove that, with respect to $\B$, every element of $W$ is 
represented by a column sign-coherent matrix in the sense of cluster algebras.
If $W$ is a finite simply laced Weyl group, each $W$-orbit of \tworoots has a highest element, analogous to
the highest root, and we calculate these elements explicitly.
We prove that if $W$ is not of affine type, the module $M$ is completely reducible in characteristic
zero and each of its nontrivial direct summands is spanned by a $W$-orbit of \tworoots.
\end{abstract}

\section*{Introduction}

A {\it \tworoot} is a symmetrized tensor product $\alpha \vee \beta := \alpha \otimes \beta + 
\beta \otimes \alpha$, where $\alpha$ and $\beta$ are orthogonal roots for a Kac--Moody algebra. In this
paper, we develop the theory of \tworoots, concentrating 
on the case where the Dynkin diagram $\Gamma$ is a Y-shaped, simply laced Dynkin diagram of rank 
$n=a+b+c=1$, with arbitrarily long branches of positive lengths $a$, $b$, and $c$. The Weyl groups
$W = W(\ypqr)$ of these types play an important role in group theory even outside the finite and affine types, 
in part because some of them have very interesting finite
quotients. For example, by adding one extra relation to the Coxeter presentation for the Weyl group of
type $\ygc{3}{4}{4}$, it is possible to obtain the group $C_2 \times {\mathbf M}$, where $C_2$ has order $2$ and 
${\mathbf M}$ is the Monster simple group \cite{ivanov93}.
The special cases of $\ygc{1}{2}{6}$ and $\ygc{1}{2}{7}$, also known respectively as $E_{10}$ and $E_{11}$, 
appear in the physics literature in M-theory and related contexts. 


\begin{figure}[ht!]
    \centering
\begin{tikzpicture}
            \node[main node] (2) {};
            \node[main node] (3) [right=1cm of 2] {};
            \node[main node] (4) [right=1.5cm of 3] {};
            \node[main node] (5) [right=1cm of 4] {};
            \node[main node] (6) [right=1cm of 5] {};
            \node[main node] (7) [right=1.5cm of 6] {};
            \node[main node] (8) [right=1cm of 7] {};
            \node[main node] (a) [above=1cm of 5] {};
            \node[main node] (b) [above=1.5cm of a] {};
            \node[main node] (c) [above=1cm of b] {};

            \path[draw]
            (2)--(3)
            (4)--(5)--(6)
            (7)--(8)
            (5)--(a)
            (b)--(c);
            \path[draw,dashed]
            (3)--(4)
            (6)--(7)
            (a)--(b);

            \draw[decorate, decoration={brace,amplitude=8pt, mirror,
            raise=6pt}] (2)--(4) node [midway, yshift=-2.5em] {$b$ vertices};

            \draw[decorate, decoration={brace,amplitude=8pt, mirror,
            raise=6pt}] (6)--(8) node [midway, yshift=-2.5em] {$a$ vertices};

            \draw[decorate, decoration={brace,amplitude=8pt, mirror,
            raise=6pt}] (a)--(c) node [midway, xshift=4em] {$c$ vertices};
\end{tikzpicture}
    \caption{The Dynkin diagram of type $Y_{a,b,c}$ }
    \label{fig:Y}
\end{figure}

The reflection representation of the Weyl group $W$ of $\ypqr$ 
is an $n$-dimensional real representation $V$ of $W$ that is equipped with a symmetric $W$-invariant bilinear 
form $B$. As we recall in Proposition \ref{virred}, the module 
$V$ turns out to be irreducible unless $\ypqr$ is one of the three affine types: affine $E_6$, $E_7$ and $E_8$.
The symmetric square $S^2(V)$ is never an irreducible module (except in trivial cases) because the kernel of $B$
(regarded as a map from $S^2(V)$ to $\R$) forms a codimension-$1$ submodule $M$ that
contains the set $\alltworoots$ of \tworoots. We call a \tworoot{} {\it real} if it arises from an orthogonal pair of
real roots. The Weyl group $W$ acts on the set $\realtworoots$ of real \tworoots in a natural way, and it follows 
from known results that the action has three orbits in type $D_4$, two orbits in type $D_n$ for $n > 4$, and one orbit 
otherwise (Proposition \ref{prop:explicitorbits}). If $\ypqr$ is not of affine type, 
then in characteristic zero, the module $M$ is a direct sum of irreducible submodules, each of which is the span of one of the 
$W$-orbits of real \tworoots (Theorem \ref{s2vdecomp}). Furthermore, in the non-affine case, the module 
$M$ has a complement in $S^2(V)$, spanned by a kind of Virasoro element (Proposition \ref{prop:virasoro}), so that $S^2(V)$ 
is completely reducible.

We show in Theorem \ref{canbas} that there is a canonically defined subset of $\realtworoots$ that forms a basis for $M$, 
which we call the canonical basis of $M$. One way to construct this basis is in terms of the stabilizer in $W$ of a simple 
root $\al_i$, which is known by work of Brink \cite{brink96} and Allcock \cite{allcock13} to be a reflection 
group with simple system $\{\be_{i,1}, \ldots, \be_{i,n-1}\}$. The canonical basis is then given (Theorem \ref{cbminimal}) 
by the (redundantly described) set $\B := \{\al_i \vee \be_{i,j} : 1 \leq i \leq n, 1 \leq j < n\}$.
If $s_i$ is a simple reflection and $v$ is a canonical basis element, then $s_i(v)$ is equal either to $-v$, or to $v$, 
or to $v + v'$ for some other basis element $v'$ (Theorem \ref{thm:canbasrefl}). This is very similar to how a simple 
reflection acts on a simple root, which is one of the reasons for the name ``\tworoots{}''.

On the module $M$, the matrices representing group elements $w \in W$ with respect to the canonical
basis have integer entries. We prove (Theorem \ref{thm:coherence}) that these matrices are column
sign-coherent in the sense of cluster algebras, which means that any two nonzero entries in the same
column of a matrix have the same sign.  Because every real \tworoot is $W$-conjugate to a basis
element (Proposition \ref{orthorbits} (iii)), an equivalent way to say this is that each real
\tworoot is an integer linear combination of canonical basis elements with coefficients of like
sign, similar to how every root of $W$ is an integer linear combination of simple roots with
coefficients of like sign. It follows that the elements of $\B$ have a simple characterization: they
are the positive real \tworoots that cannot be expressed as a positive linear combination of two or
more positive real \tworoots.

We use the canonical basis $\B$ to define a partial order $\leq_2$ on $\realtworoots$ 
by declaring that $v_1\leq_2 v_2$ if $v_2-v_1$ is a positive linear combination of elements of $\B$. 
We prove in Proposition \ref{prop:refinement} that $\leq_2$ is a refinement of the so-called monoidal partial order 
defined by Cohen, Gijsbers, and Wales on sets of orthogonal positive roots in \cite{cohen06}.
In the case where $W$ is finite, it then follows (Theorem \ref{thm:highest}) that $\realtworoots$ contains a unique maximal 
\tworoot with respect to $\leq_2$, which we describe explicitly (Theorem \ref{thm:highestlist}). 

Although we concentrate on the case of type $\ypqr$ in this paper, some of the results hold for type $A_n$ by
restriction.  
The difference in type $A$ is that the orthogonal complement of a root is not spanned by the roots it contains. 
When $V$ is the reflection representation in type $A$ (corresponding to the partition $(n-1, 1)$), it is known that 
$S^2(V)$ is the direct sum of three representations, corresponding to the partitions $(n)$, $(n-1, 1)$ and $(n-2, 2)$. 
(This follows from \cite[Example 2]{bowman15}, using the fact that the exterior square $\exterior{2}(V)$ corresponds to 
the partition $(n-2, 1^2)$; see also \cite[Proposition 5.4.12]{geck00}.) In this case, the submodule $(n-2, 2)$ corresponds
to the span of the \tworoots, the submodule $(n-1, 1)$ corresponds to its complement in the module $M$, and the submodule $(n)$
corresponds to the Virasoro element. 

We also note that the results of this paper do not seem to generalize readily to all simply
laced Weyl groups. For example, let $n> 6$ and consider the simply laced Weyl group $W=W(\tilde
D_{n-1})$ of type affine
$D$ and rank $n$. Then by the third example in \cite[Section 4]{allcock13}, for any simple root
$\alpha_i$ of $W$ the stabilizer $W_{\alpha_i}$ of $\alpha_i$ in $W$ is a Weyl group of rank
$n$, not $n-1$. It follows that the elements of the form $\alpha_i\vee\beta$ where $\beta$ is a simple root of
$W_{\alpha_i}$ no longer form a linearly independent set, therefore the conclusions in Theorem
\ref{cbminimal} no longer hold.


The paper is organized as follows. 
Section \ref{sec:canbas} defines the canonical basis $\B$ of \tworoots (Theorem \ref{canbas}).
Section \ref{sec:stab} explains how to construct the basis $\B$ in terms of the stabilizers of real roots 
(Theorem \ref{cbminimal}).
Section \ref{sec:orbits} describes the $W$-orbits of real \tworoots.
Section \ref{sec:reflact} describes the action of reflections on \tworoots, and gives a simple formula
(Theorem \ref{thm:canbasrefl}) for the action of a simple reflection on a canonical basis element.
In Section \ref{sec:coherence}, we prove the sign-coherence properties of the canonical basis
({Theorem \ref{thm:coherence} and Theorem \ref{thm:coherence2}).
In Section \ref{sec:highest}, we prove that a $W$-orbit of \tworoots for a finite simply laced Weyl group has a 
unique maximal element (Theorem \ref{thm:highest}) and we determine this maximal element explicitly
(Theorem \ref{thm:highestlist}). 
In Section \ref{sec:repthy}, we use $W$-orbits of \tworoots to describe the submodules of $S^2(V)$, both in
general characteristic (Theorem \ref{modradical}) and in characteristic zero (Theorem \ref{s2vdecomp}).
In Section \ref{sec:faithful}, we determine when $W$ acts faithfully on the representations arising from $W$-orbits
(Theorem \ref{thm:faithful}).
The results in this paper immediately suggest directions for future research, which we summarize in the conclusion.

\section{The canonical basis of \tworoots}\label{sec:canbas}

Throughout this paper, we will work over a field $F$ that is of characteristic zero unless otherwise stated. By default, 
we will assume that $F = \R$, but everything will be defined over $\Q$, and scalars can be extended if necessary.

Let $\Gamma = \ypqr$ be a simply laced Dynkin diagram with $n=a+b+c+1$ vertices, consisting of three paths with $a$, $b$, and 
$c$ vertices emanating from a trivalent branch vertex. Let $A$ be the associated Cartan matrix, whose entries $A_{ij}$ are 
equal to $2$ if $i = j$, $-1$ if $i$ and $j$ are adjacent in $\Gamma$, and $0$ otherwise.

Let $W$ be the Weyl group associated to $\Gamma$. It is generated by the set $S=\{s_1, \dots, s_n\}$
indexed by the vertices of $\Gamma$, and subject to the defining relations $s_i^2 = 1$, $s_i s_j =
s_j s_i$ if $A_{ij}=0$, and $s_i s_j s_i = s_j s_i s_j$ if $A_{ij}=-1$.

Let $\simproots = \{\al_1, \dots, \al_n\}$ be the set of simple roots of $W$, and let $V$ be the $\R$-span of $\simproots$.
Let $B$ be the Coxeter bilinear form on $V$, normalized so that $B(\al_i, \al_j) = A_{ij}$, 
and let $$
V^\perp = \{v \in V : B(v, v') = 0 \mbox{\ for\ all\ } v' \in V\}
$$ be the radical of $B$.
A real root of $W$ is an element of $V$ of the form $w(\al_i)$, where $w \in W$ and $\al_i$ is a simple root.  The real 
root $\al = w(\al_i)$ is associated with the reflection $s_{\al} = ws_iw^{-1}$; in particular,
we have $s_{\alpha_i}=s_i$ for each $i$. The reflection $s_\alpha$ acts on basis elements of $V$ by the 
formula $$
s_{\al}(\al_j) = \al_j - B(\al, \al_j) \al
.$$ 
When $V$ is endowed with this action, we call $V$ the {\it reflection representation} of $W$.

It is immediate from the above formula that $W$ stabilizes the $\Z$-span, $\Z \simproots$, of the simple roots.
The lattice $\Z \simproots$ is called the {\it root lattice} and is often denoted by $Q$.
The form $B$ is invariant under this action of the Weyl group, meaning that we always have $B(v,v')=B(w.v,w.v')$. This 
implies that $V^\perp$ is a $W$-submodule of $V$, and that we have $B(\al, \al)=2$ for every real root $\al$.

A Kac--Moody algebra may have roots other than real roots; such roots are called {\it imaginary roots}. We will not give 
the full definition of imaginary roots, but we will need the result that in the case of affine Kac--Moody algebras, the 
imaginary roots are precisely the nonzero integer multiples $n \delta$ of the lowest positive imaginary root, $\delta$.
The root $\delta$ satisfies $B(\delta, v)=0$ for all $v \in V$.

\begin{prop}\label{virred}
If $W$ is a Weyl group of type $\ypqr$, then $V$ is an irreducible $W$-module if and only 
if $W$ is not of type affine $E_6$, affine $E_7$ or affine $E_8$.
\end{prop}

\begin{proof}
By \cite[Proposition 6.3]{humphreys90}, it suffices to show that $V^\perp=0$. We omit the rest of the proof,
because the result is well known; see for example \cite[Example 4.3]{dolgachev08}.
\end{proof}

We regard the symmetric square, $S^2(V)$ of $V$ as a submodule (rather than as a quotient) of $V \otimes V$. 
If $\al, \be \in V$, we
write $\al \vee \be$ (or $\be \vee \al$) for the element of $S^2(V)$ given by $\al \otimes \be + \be \otimes \al$. 
The basis $\simproots$ of $V$ gives rise to a basis of $S^2(V)$ given by $$
\{\al_s \vee \al_t : s, t \in \simproots\}
,$$ which we call the {\it standard basis} of $S^2(V)$ (with respect to $\simproots$). Restricting the diagonal 
action of $W$ on $V \otimes V$ gives $S^2(V)$ the structure of a $W$-module.

The following result is an immediate consequence of the $W$-invariance of $B$. 

\begin{lemma}\label{codimone}
Regard $B$ as a map $B : S^2(V) \ra F$, and let $M = \ker B$. Then $M$ is a $W$-submodule of $S^2(V)$ of
dimension $$
\dim(M) = \dim(S^2(V)) - 1 = \binom{n+1}{2} - 1
,$$ and $S^2(V)/M$ affords the trivial representation of $W$. \qed
\end{lemma}

Recall that the positive roots of $W$ are partially ordered in such a way that $\al \leq \be$ if and only if 
$\be - \al$ is a nonnegative linear combination of simple roots, and that if $W$ is finite, then there exists a 
highest root with respect to this order. In type $A_n$, the highest root is the sum of all the simple roots. 
In type $D_n$, the highest root is $\sum_{i=1}^n \lambda_i \al_i$, where we have $$
\lambda_i = \begin{cases}
1 & \text{if}\ i \text{\ is an endpoint of\ } \Gamma;\\
2 & \text{otherwise.}\\
\end{cases}
$$ 

\begin{defn}\label{minroot}
We define a positive real root $\al$ of type $\ypqr$ to be {\it \elementary} if $\al$ is a simple root, or the 
highest root in a type $A_3$ standard parabolic subsystem, or the highest root in a type $D_m$ standard 
parabolic subsystem for $m \geq 4$. To each \elementary root $\al$, 
we associate a nonempty subset $L(\al)$ of the vertices of $\Gamma$, defined as follows.
\begin{itemize}
    \item[(1)]{If $\al_i$ is a simple root, then we define $L(\al_i)$ to be the set of all $j$ for which 
    $A_{ij} = 0$; in other words, the set of all vertices in $\Gamma$ that are not equal to or adjacent to $i$.}
    \item[(2)]{If $\al$ is the highest root in a standard parabolic subgroup of type $A_3$, then
    $\al = \al_i + \al_j + \al_k$ for some path $i$--$j$--$k$ in $\Gamma$, and we define 
    $L(\al) = \{j\}$. We denote $\al_i + \al_j + \al_k$ by $\eta_{i, k}$. If $j$ is not the 
    branch point of $\Gamma$, then $i$ and $k$ can be deduced from a knowledge of $j$, and we 
    may write $\eta_j$ for $\eta_{i,k}$.}
    \item[(3)]{Suppose that $\al$ is the highest root of a standard parabolic subgroup of type
    $D_m$. If $m=4$, we define $L(\al)$ to be the three-element set consisting of the neighbours of the
    branch point in $\Gamma$. If $m > 4$, we define $L(\al)$ to be the single element $\{i\}$ indexing
    the unique simple root in the support of $\al$ that is maximally far from the branch point. In either
    case, we may denote $\al$ by $\theta_i$ for any $i \in L(\al)$.}
\end{itemize}
We say that an \elementary root $\al$ is of {\it type 1, 2, or 3}, depending on which of the three mutually exclusive
conditions above applies.  If $\al$ is an \elementary root and $i \in L(\al)$, then we say $\al$ is 
{\it \elementary with respect to $\al_i \in \simproots$}.
\end{defn}

\begin{rmk} \label{uniquetyped}
Any simple root $\al_i$ other than the one corresponding to the branch point of $\Gamma$ lies in a
unique standard parabolic subsystem of type $D_m$ that is of minimal rank. The simple roots involved in this 
parabolic subsystem are those on the path between $\al_i$ and the branch point, together with all the neighbours
of the branch point. In the notation of Definition \ref{minroot}, part (3), the highest root of this parabolic 
subsystem is $\theta_i$, and it is elementary with respect to $i$.
\end{rmk}

\begin{lemma}\label{minrootcount}
Let $\al_i$ be a simple root of type $\ypqr$, and let $n = a+b+c+1$. Then there are precisely $n-1$ \elementary roots
that are \elementary with respect to $\al_i$, and each such \elementary root $\al$ satisfies $B(\al_i, \al) = 0$.
\end{lemma}

\begin{proof}
A case-by-case check based on Definition \ref{minroot} shows that whenever $\al$ is a simple root and 
$i \in L(\al)$, we have $B(\al_i, \al) = 0$.  For the other assertion, we consider three cases, according 
as $\al_i$ is an endpoint of the Dynkin diagram $\Gamma$, or the branch point, or one of the other vertices.

If $\al_i$ is an endpoint of $\Gamma$, then the $n-1$ \elementary roots $\al$ with $i \in L(\al)$ are (a) 
the $n-2$ simple roots that are not equal or adjacent to $\al_i$, and (b) the root $\theta_i$ of Definition 
\ref{minroot}, part (3).

If $\al_i$ is the branch point of $\Gamma$, then the $n-1$ \elementary roots $\al$ with $i \in L(\al)$ are (a) 
the $n-4$ simple roots that are not equal or adjacent to $\al_i$, and (b) the three \elementary roots 
$\eta_{h,j}$ of type $2$ with $i \in L(\eta_{h,j})$.

If $\al_i$ is neither an endpoint nor the branch point of $\Gamma$, then the $n-1$ \elementary roots $\al$ with
$i \in L(\al)$ are (a) the $n-3$ simple roots that are not equal or adjacent to $\al_i$; (b) the root $\eta_i$ 
of Definition \ref{minroot}, and (c) the root $\theta_i$ of Definition \ref{minroot}, part (3).
\end{proof}

Recall that the roots of $W$ are partially ordered by stipulating that $\al \leq \be$ if $\be - \al$ is a linear
combination of simple roots with nonnegative coefficients. 

\begin{lemma}\label{minismin}
Let $\al$ be a positive root that is \elementary with respect to the simple root $\al_i$ in type $\ypqr$. 
If $\al$ is a linear combination of positive roots $\be_1, \ldots, \be_r$ with positive 
integer coefficients and with $r > 1$, then not all of the $\be_k$ can be orthogonal to $\al_i$.
\end{lemma}

\begin{proof}
Note that the hypotheses imply that we have $\be_k < \al$ for all $k$.
If $\al$ is a simple root, then the statement holds vacuously. 

If $\al = \eta_{h,j} = \al_h + \al_i + \al_j$, then the only positive roots $\be < \al$ are $$
\be \in \{\al_h, \al_i, \al_j, \al_h + \al_i, \al_i + \al_j\}
,$$ and none of the roots in this list is orthogonal to $\al_i$.

Finally, suppose that $\al = \theta_i$, and consider the parabolic subgroup of type $D_m$ in which $\theta_i$ is 
the highest root. Define $\al_j$ to be the simple root adjacent to $\al_i$ in the support of $\theta$. 
It is well known (and mentioned in \cite[\S4]{allcock13}) that the roots orthogonal to $\al_i$ in $D_m$ 
form a root system of type $D_{m-2} \cup A_1$, if we interpret $D_3$ as $A_3$ and $D_2$ as $A_1 \cup A_1$.
The roots in the $D_{m-2}$ component are those that do not involve $\al_i$ or $\al_j$, and the roots in the
$A_1$ component are $\{\pm \theta_i\}$. It follows that the only positive roots $\be < \theta_i$ that are 
orthogonal to $\al_i$ come from the $D_{m-2}$ component, so that $\al_i$ and $\al_j$ both appear
with zero coefficient in every $\be_k$. This contradicts the fact that $\al_i$ appears with a nonzero coefficient 
in $\al$.
\end{proof}

\begin{defn}\label{def:canbas}
Let $V$ be the reflection representation associated with the Dynkin diagram $\ypqr$. We define $\B = \B(a, b, c)$ 
to be the subset of $S^2(V)$ consisting of all elements of the form $\al_i \vee \be$, where $\al_i \in \simproots$ 
and where $\be$ is \elementary with respect to $\al_i$. 
\end{defn}

\begin{theorem}\label{canbas}
Let $\Gamma$ be a Dynkin diagram of type $\ypqr$, let $n=a+b+c+1$, and let $B$ be the bilinear form on the
associated reflection representation $V$. The set $\B = \B(a,b,c)$ is a basis for the submodule $M = \ker B$ of 
$S^2(V)$.
\end{theorem}

\begin{proof}
Lemma \ref{minrootcount} implies that every element of $\B$ lies in $M$.
The proof now reduces to showing that $\B = \B(a,b,c)$ is linearly independent and has cardinality
$\binom{n+1}{2} - 1$, which by Lemma \ref{codimone} is equal to $\dim(M)$.
We prove these two claims by induction on $k = \max(a, b, c)$. 

The base case, $k=1$, corresponds to $\Gamma$ being of type $D_4$.
We label the vertices of $\Gamma$ by $\{1, 2, 3, 4\}$, where $2$ is the branch point. The canonical basis is then given
by $$
\{  
\al_1 \vee \theta_1, \ \al_3 \vee \theta_3, \ \al_4 \vee \theta_4, \ 
\al_2 \vee \eta_{1,3}, \ \al_2 \vee \eta_{1,4}, \ \al_2 \vee \eta_{3,4}, \ 
\al_1 \vee \al_3, \ \al_1 \vee \al_4, \ \al_3 \vee \al_4
\}
,$$ which has size $9 = \binom{5}{2} - 1$, as required. 

Suppose for a contradiction that there is a nontrivial dependence relation between these nine elements.
We can show that $\B$ is linearly independent by expanding
everything in terms of the standard basis $\{\al_i \vee \al_j : 1 \leq i, j \leq n\}$ of $S^2(V)$, as follows.
For each $i \in \{1, 3, 4\}$, the only element of $\B$ with $\al_i \vee \al_i$ in its support is $\al_i \vee \theta_i$.
This implies that the elements $\al_i \vee \theta_i$ for $i \in \{1, 3, 4\}$ appear with coefficient zero in the 
dependence relation. Next, equating coefficients of $\al_1 \vee \al_2$ implies that $\al_2 \vee \eta_{1,3}$ and 
$\al_2 \vee \eta_{1,4}$ occur with equal and opposite coefficients in the dependence relation. Extending this
argument to all standard basis elements $\al_i \vee \al_2$ for $i \in \{1, 3, 4\}$ implies that all basis elements 
$\al_2 \vee \eta_{1,3}$, $\al_2 \vee \eta_{1,4}$ and $\al_2 \vee \eta_{3,4}$ occur with coefficient zero in the 
dependence relation. The remaining elements of $\B$, $\al_1 \vee \al_3$, $\al_1 \vee \al_4$ and $\al_3 \vee \al_4$,
are all standard basis elements and are therefore linearly independent, completing the base case.

For the inductive step, we will prove that the statements hold when $\max(a,b,c) = k+1$, assuming that
they hold when $\max(a,b,c)=k$. Since we now have $\max(a,b,c) > 1$, it follows that $n=a+b+c+1 > 4$.
We assume without loss of generality that $a \leq b \leq c$. Denote the vertex of $\Gamma$ at 
the end of the $c$-branch by $1$, and denote the vertex next to it by $2$; note that the hypothesis $n > 4$
guarantees that $2$ is not the branch point of $\Gamma$. Let $V'$ be the reflection representation 
in type $\ygc{a}{b}{c-1}$, so that the set $\simproots \backslash \{\al_1\}$ is a basis for $V'$, and let 
$\B'=\B(a,b,c-1)$, so that $\B' \subset \B$. The elements of $\B \backslash \B'$ are $\al_1 \vee \theta_1$, 
$\al_2 \vee \eta_2$, and the $n-2$ elements $\al_1 \vee \al_j$ for $j \not\in \{1, 2\}$. This implies that
$|\B| = |\B'|+n$, and therefore by induction that $$
|\B| = |\B'| + n = \binom{n}{2} - 1 + n = \binom{n+1}{2} - 1
,$$ as required.

It remains to show that $\B$ is linearly independent. If not, then the linear independence of $\B'$ (by induction)
means that we must have $$
\sum_{b_i \in \B \backslash \B'} \lambda_i b_i = \sum_{b'_j \in \B'} \mu_i b'_j
,$$ for some scalars $\lambda_i$ and $\mu_j$, where some $\lambda_i$ is nonzero. 
Now express both sides of this equation with respect to the standard
basis of $S^2(V)$, so that the right hand side is a linear combination of the standard basis of $S^2(V')$.
The only element of $\B$ with a nonzero coefficient of $\al_1 \vee \al_1$ is $\al_1 \vee \theta_1$, so equating
coefficients of $\al_1 \vee \al_1$ in the above equation implies that $\al_1 \vee \theta_1$ appears with coefficient
zero. The only elements of $\B$ with a nonzero coefficient of $\al_1 \vee \al_2$ are $\al_1 \vee \theta_1$ and
$\al_2 \vee \eta_2$, so equating coefficients of $\al_1 \vee \al_2$ implies that $\al_2 \vee \eta_2$ appears with
coefficient zero. The other elements of $\B \backslash \B'$ are all standard basis elements that do not lie
in $S^2(V')$, so they also appear with coefficient zero. This contradiction completes the proof.
\end{proof}

\section{Stabilizers of real roots}\label{sec:stab}

Recall that a root of $W$ is called {\it real} if it is $W$-conjugate to a simple root. We denote the set of real roots
of $W$ by $\realroots$. In Section \ref{sec:stab}, we describe the relationship between the basis $\B$ and the stabilizers 
of the real roots of $W$. 

To do so, it is helpful to introduce some graph theoretic terminology. For each
integer $k \geq -1$, there is a notion of attaching a path of length $k$ to a graph $G$ with $n$ vertices to form a 
graph with $n+k$ vertices.

\begin{defn}\label{def:attach}
Let $k \geq 1$ be an integer, and let $\al$ be a vertex of a graph $G$. To {\it attach a path of length 
$k$ to $G$ at $\al$}, we take the disjoint union of $G$ and a path $P$ with $k$ vertices, and then 
add an edge between $\al$ and one of the endpoints of $P$.

To {\it attach a path of length $0$ to $G$ at $\al$}, we simply take the graph $G$ itself. To
{\it attach a path of length $-1$ to $G$ at $\al$}, we remove the vertex $\al$ and all edges incident to $\al$.
\end{defn}

\begin{defn}\label{def:hpqr}
Let $a, b, c \geq -1$ be integers, and let $H$ be the 6-cycle $h_1$--$h_2$--$h_3$--$h_4$--$h_5$--$h_6$--$h_1$.
We define the $\hexg{a}{b}{c}$ to be the graph obtained by
attaching paths of lengths $a$, $b$, and $c$ to $H$ at the vertices $h_1$, $h_3$, and $h_5$, respectively.
\end{defn}


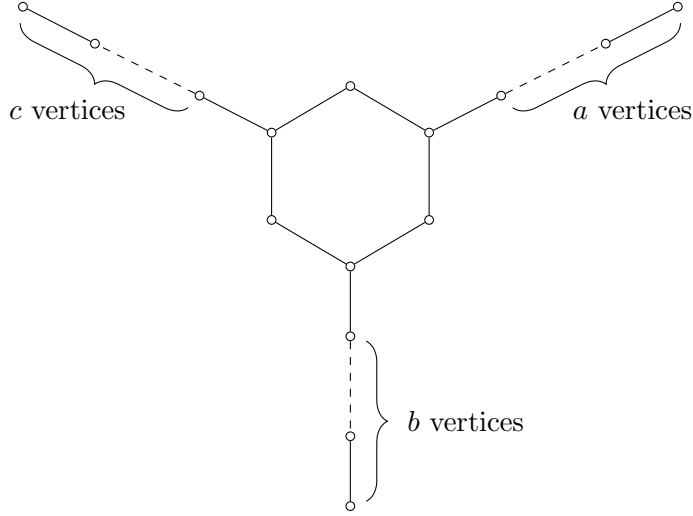
\begin{figure}[ht!]
    \centering
\begin{tikzpicture}
    \node (c) {};
            \node[main node] (1) [above right=0.4cm and 0.866cm of c] {};
            \node[main node] (x1) [above right=0.4cm and 0.866cm of 1] {};
            \node[main node] (y1) [above right=0.6cm and 1.299cm of x1] {};
            \node[main node] (z1) [above right=0.4cm and 0.866cm of y1] {};
            \node[main node] (2) [below right=0.4cm and 0.866cm of c] {};
            \node[main node] (3) [below = 1cm of c] {};
            \node[main node] (x3)[below =0.8cm of 3] {};
            \node[main node] (y3)[below =1.2cm of x3] {};
            \node[main node] (z3)[below =0.8cm of y3] {};
            \node[main node] (4) [below left=0.4cm and 0.866cm of c] {};
            \node[main node] (5) [above left=0.4cm and 0.866cm of c] {};
            \node[main node] (x5) [above left=0.4cm and 0.866cm of 5] {};
            \node[main node] (y5) [above left=0.6cm and 1.299cm of x5] {};
            \node[main node] (z5) [above left=0.4cm and 0.866cm of y5] {};
            \node[main node] (6) [above = 1cm of c] {};

            \path[draw]
            (1)--(2)--(3)--(4)--(5)--(6)--(1)
            (1)--(x1)
            (y1)--(z1)
            (3)--(x3)
            (y3)--(z3)
            (5)--(x5)
            (y5)--(z5);

            \path[draw,dashed]
            (x1)--(y1)
            (x3)--(y3)
            (x5)--(y5);

            \draw[decorate, decoration={brace,amplitude=8pt, mirror,
            raise=6pt}] (x1)--(z1) node [midway, xshift=1.5em, yshift=-2em] {$a$ vertices};

            \draw[decorate, decoration={brace,amplitude=8pt,
            raise=6pt}] (x3)--(z3) node [midway, xshift=4em] {$b$ vertices};

            \draw[decorate, decoration={brace,amplitude=8pt,
            raise=6pt}] (x5)--(z5) node [midway, xshift=-1.5em, yshift=-2em] {$c$ vertices};
\end{tikzpicture}
    \caption{The graph $H_{a,b,c}$ }
    \label{fig:H}
\end{figure}

\begin{rmk}\label{rmk:qabc}
The graphs $\hexg{a}{b}{c}$ are denoted by $Q_{a+1, b+1, c+1}$ in the ATLAS of Finite Groups \cite[pp 232--233]{conway85},
where they play an important role in the structure of the Monster simple group.
\end{rmk}

\begin{lemma}\label{lem:componentcount}
The number of connected components of $\hexg{a}{b}{c}$, where $a \leq b \leq c$, is $3$ if
$a=b=c=-1$, is $2$ if $a=b=-1$ and $c \geq 0$, and is $1$  otherwise. 
\end{lemma}

\begin{proof}
This follows from the definition of $\hexg{a}{b}{c}$. 
\end{proof}

In the case of the Dynkin diagram $\ypqr$, the stabilizer in $W$ of a real root has been determined explicitly by
Allcock \cite{allcock13}, using a result of Brink \cite{brink96}.

\begin{theorem}[Allcock, Brink]\label{thm:hpqr}
Let $W$ be a Weyl group of type $\ypqr$ with $a, b, c \geq 1$, and let $\al$ be a real root of $W$. Then the stabilizer 
$W_\al = \Stab_W(\al)$ of $\al$ in $W$ is generated by the reflections it contains, and $W_\al$ is a simply laced Weyl 
group of type $\hpqr$.
\end{theorem}

\begin{proof}
Since $\ypqr$ is simply laced, all real roots are $W$-conjugate to each other and therefore have conjugate stabilizers. 
It therefore suffices to prove the theorem in the case where $\al$ is a simple root $\al_i$ associated to a 
Coxeter generator $s \in W$.

It follows from the main result of \cite{brink96} (see also \cite[Corollary 7]{allcock13}) that $W_{\al_i}$ can be expressed
as a semidirect product $W_\Omega \rtimes \Gamma_\Omega$, where $W_\Omega$ is the subgroup generated by all the reflections
that fix $\al_i$, and $\Gamma_\Omega$ is the free group $\pi_1(\Delta^{\text{\rm odd}}, s)$. Since every edge in $\ypqr$ has
an odd label of $3$, the graph $\Delta^{\text{\rm odd}}$ is simply the Dynkin diagram $\Gamma$. The connected component of
$\Gamma$ containing $s$ has no circuits, which means that the free group in question is trivial, and that 
$W_\al \cong W_\Omega$.

The proof is completed from the discussion following \cite[Theorem 13]{allcock13}, which describes an equivalent construction
of the graphs $\hpqr$ as the graphs of $W_\Omega$.
\end{proof}

\begin{exa}\label{exa:hexamples}
Let $W = W(\ypqr)$ and let $\al$ be a real root of $W$. If $W$ is a Weyl group of type 
$D_4 = \ygc{1}{1}{1}$,
$D_5 = \ygc{1}{1}{2}$, or 
$E_8 = \ygc{1}{2}{4}$, 
then by Theorem \ref{thm:hpqr} the corresponding Dynkin diagrams $\hexg{a-2}{b-2}{c-2}$ for $W_\al$ are as pictured from left 
to right in Figure \ref{fig:Hexamples}. If $W$ is the affine Weyl group of type $\widetilde{E}_6 = \ygc{2}{2}{2}$, then the 
corresponding Dynkin diagram $\hexg{a-2}{b-2}{c-2} = \hexg{0}{0}{0}$ for $W_\al$ is simply a hexagon, which equals the Dynkin 
diagram of type $\widetilde{A}_5$. In other words, the stabilizer of each real root in type affine $E_6$ is isomorphic (as a
reflection group) to the Weyl group of type affine $A_5$.
\end{exa}

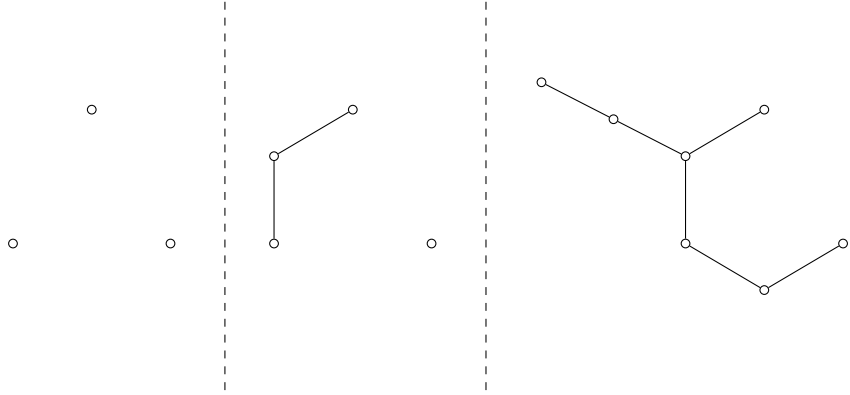
\begin{figure}[ht!]
    \centering
\begin{tikzpicture}
    \node (c) {};
            \node[main node] (2) [below right=0.4cm and 0.866cm of c] {};
            \node[main node] (4) [below left=0.4cm and 0.866cm of c] {};
            \node[main node] (6) [above = 1cm of c] {};

            \node(a) [above right=2.5cm and 1.5cm of c] {};
            \node(b) [below right=2.5cm and 1.5cm of c] {};

            \node (1c) [right=3.2cm of c] {};
            \node[main node] (12) [below right=0.4cm and 0.866cm of 1c] {};
            \node[main node] (14) [below left=0.4cm and 0.866cm of 1c] {};
            \node[main node] (15) [above left=0.4cm and 0.866cm of 1c] {};
            \node[main node] (16) [above = 1cm of 1c] {};

            \node(x) [above right=2.5cm and 1.5cm of 1c] {};
            \node(y) [below right=2.5cm and 1.5cm of 1c] {};

            \node (2c) [right=5.2cm of 1c] {};
            \node[main node] (22) [below right=0.4cm and 0.866cm of 2c] {};
            \node[main node] (23) [below = 1cm of 2c] {};
            \node[main node] (24) [below left=0.4cm and 0.866cm of 2c] {};
            \node[main node] (25) [above left=0.4cm and 0.866cm of 2c] {};
            \node[main node] (2x5) [above left=0.4cm and 0.866cm of 25] {};
            \node[main node] (2y5) [above left=0.4cm and 0.866cm of 2x5] {};
            \node[main node] (26) [above = 1cm of 2c] {};

            \path[draw]
            (14)--(15)--(16)
            (22)--(23)--(24)--(25)--(26)
            (25)--(2x5)--(2y5);

            \path[draw,dashed] (a)--(b)
            (x)--(y);

\end{tikzpicture}
\caption{The $H$-diagrams corresponding to the Weyl groups $D_4, D_5$ and
$E_8$}
    \label{fig:Hexamples}
\end{figure}

\begin{theorem}\label{cbminimal}
Let $W$ be a Weyl group of type $\ypqr$ with $a, b, c \geq 1$, and let $n = a+b+c+1$ be the rank of $W$.
\begin{itemize}
\item[\rm (i)]{Let $\al_i$ be a simple root of $W$, and regard
the stabilizer $W_{\al_i}$ as a Weyl group as in Theorem \ref{thm:hpqr}. Then the simple roots of $W_{\al_i}$ are
precisely the $(n-1)$ \elementary roots with respect to $i$ from Definition \ref{minroot}.}
\item[\rm (ii)]{The canonical basis $\B(a,b,c)$ consists of all elements $\al_i \vee \be$,
where $\al_i$ is a simple root of $W$, and $\be$ is a simple root of the stabilizer $W_{\al_i}$.}
\end{itemize}
\end{theorem}

\begin{proof}
By Theorem \ref{thm:hpqr}, the group $W_{\alpha_i}$ has $(a+b+c)$ simple roots. In general, the simple roots can be 
characterized as the roots that are not expressible as positive integer linear combinations of other positive roots.
The elementary roots with respect to $i$ have this property by Lemma \ref{minismin}, and there are $n-1=a+b+c $
of them by Lemma \ref{minrootcount}. The conclusion of (i) follows.

The assertion of (ii) follows from (i) and Theorem \ref{canbas}.
\end{proof}

\begin{cor}\label{cor:starop}
Let $i$ and $j$ be adjacent vertices of $W$, and for $k \in \{i, j\}$, define $R_k \subset \B$ be the 
set of basis elements of the form $\al_k \vee \be$. Then the map $\phi_{ij} : R_i \rightarrow R_j$ defined
by $\phi_{ij}(v) = s_i s_j (v)$ is a well-defined bijection with inverse $\phi_{ji}$.
\end{cor}

\begin{proof} 
Note that for any real root $\be$ we have $$
\phi_{ij}(\al_i \vee \be) = (s_i s_j (\al_i)) \vee (s_i s_j(\be)) = \al_j \vee s_i s_j(\be)
.$$ In this way, $\phi_{ij}$ induces a bijection $\phi'_{ij}$ between the real roots $\be$ orthogonal to $\al_i$ and 
the real roots $\be' = s_i s_j(\be)$ orthogonal to $\al_j$. Because the reduced word $s_i s_j$ has a length of $2$, it
makes precisely two positive roots negative. These are 
$\al_j$ and $s_i(\al_j) = \al_i + \al_j$, neither of which is orthogonal to $\al_i$. It follows that
$\phi'_{ij}$ sends positive roots to positive roots, and negative roots to negative roots. In turn, this implies that
$\phi'_{ij}$ sends the simple roots of $W_{\al_i}$ to the simple roots of $W_{\al_j}$, which proves that $\phi_{ij}$
has the claimed property by Theorem \ref{cbminimal} (i). The claim about inverses is immediate from the fact that 
$s_i s_j$ is the inverse of $s_j s_i$.
\end{proof}

We record below a technical lemma for future use. 

\begin{lemma}\label{abgcartan}
Let $\al_i$ be a simple root, let $\al_i \vee \be \in \B(a,b,c)$ be a canonical basis element, and let 
$\gamma \in \simproots \backslash \{\al_i, \be\}$ be another simple root.
\begin{enumerate}
    \item  
At least one of the following holds: 
\begin{itemize}
    \item[\text{\rm (i)}]{$B(\al_i, \gamma) = B(\be, \gamma) = 0$;}
    \item[\text{\rm (ii)}]{$B(\al_i, \gamma)=-1$;}
    \item[\text{\rm (iii)}]{$B(\al_i, \gamma) = 0$ and $B(\be, \gamma)=-1$.}
\end{itemize}
\item We have $B(\be, \gamma) \in \{-1, 0, 1\}$.  
\end{enumerate}
\end{lemma}

\begin{proof}
    (1) Since $\gamma \ne \al_i$, it follows from the definition of the generalized Cartan matrix that we must
have $B(\al_i, \gamma) \in \{0, -1\}$. If $B(\al_i, \gamma)=-1$ then (ii) holds and we are done, so assume that we have 
$B(\al_i, \gamma)=0$.

It now follows from Theorem \ref{cbminimal} that $\gamma$ and $\be$ are both simple roots of $W_{\al_i}$. Consideration of
the generalized Cartan matrix of $W_{\al_i}$ now shows that either $B(\be, \gamma)=0$ or $B(\be, \gamma)=-1$, which completes the proof.

(2) From the explicit description of $\B$, the root $\be$ is either a simple root, or is the highest root in a parabolic subsystem of $\ypqr$ of
type $A_3$ or $D_m$. Note that if $\al'$ is a simple root that occurs with coefficient $c \geq 2$ in $\be$, then it must be the case that
$\be$ is the highest root in a subsystem of type $D_n$ and $c = 2$. In this case, the only simple roots $\gamma$
in $\ypqr$ that are adjacent to $\al'$ must also be in the support of $\be$.

Suppose first that $\gamma$ is not in the support of $\be$. If $\gamma$ is not adjacent to a simple root in the support of $\be$, then we
have $B(\be, \gamma) = 0$, which satisfies the conclusion. If, on the other hand, $\gamma$ is adjacent to a simple root $\al'$ in the 
support of $\be$, then
the previous paragraph shows that $\gamma$ is adjacent to a simple root $\al'$ in the support of $\be$ that occurs with
coefficient $1$. There is a unique such simple root $\al'$, because the support of $\be$ is a tree and there are no circuits in the 
subgraph of $\ypqr$ consisting of $\gamma$ and the support of $\be$. It follows that $B(\be, \gamma) = -1$ in this case.

The final possibility is that $\gamma$ is in the support of $\be$. In this case, $\gamma$ and $\beta$ lie in a subsystem of type $A_3$ or $D_m$, 
and a case-by-case check (depending on whether the subsystem is of type $A_3$, $D_4$, or $D_m$ where $m > 4$) shows that 
$B(\be, \gamma)\in \{0,1\}$.
\end{proof}

\section{Orbits of \tworoots}\label{sec:orbits}

In Section \ref{sec:orbits}, we investigate the action of $W$ on pairs of orthogonal roots in more detail 
(Proposition \ref{orthorbits}), which leads to a detailed description of the $W$-orbits of \tworoots 
(Proposition \ref{prop:explicitorbits}). The following result is well known, and follows for example from \cite[Lemma 3.6]{brady02}.

\begin{lemma}\label{lem:simprootorbits}
Let $W$ be a simply laced Weyl group, and let $\al_i$ and $\al_j$ be two Coxeter generators of $W$. Then $\al_i$ and $\al_j$ are 
conjugate in $W$ if and only if they lie in the same connected component of the Dynkin diagram of $W$.
\qed\end{lemma}

The next result will be used in the proof of Proposition \ref{orthorbits} below.

\begin{lemma}\label{lem:branchnodes}
Let $\Gamma$ be a Dynkin diagram of type $\ypqr$ and let $\Gamma'$ be a parabolic subsystem of type $D_m$ for $m \geq 4$. 
Number the vertices of $\Gamma'$ such that $\be_0$ is the branch vertex, and such that the paths are 
$\be_0$--$\be'$, $\be_0$--$\be''$, and $\be_0\text{--}\be_1\text{--}\cdots\text{--}\be_{m-3}$. Let $\theta$ be the highest 
root of $\Gamma'$. Then the ordered pairs $(\be_{m-3}, \theta)$ and $(\be', \be'')$ are in the same $W(D_m)$-orbit.
\end{lemma}

\begin{proof}
Let $s_i$, $s'$, $s''$ be the reflections associated to the roots $\be_i$, $\be'$ and $\be''$, respectively. Direct calculation shows
that $$
(s_{m-4}s_{m-3})(s_{m-5}s_{m-4})\cdots(s_0 s_1)(s' s_0)\big( (\be', \be'') \big) = (\be_{m-3}, \theta)
,$$ which completes the proof.
\end{proof}

\begin{prop}\label{orthorbits}
Let $W$ be a Weyl group of type $\ypqr$. 
\begin{itemize}
\item[\text{\rm (i)}]{The group $W$ acts transitively on $\realroots$.}
    \item[\text{\rm (ii)}]{Every ordered pair of orthogonal real roots of $W$ is $W$-conjugate to a pair of orthogonal simple 
    roots of $W$.} 
    \item[\text{\rm (iii)}]{Every real \tworoot is $W$-conjugate to an element of $\B$.}
    \item[\text{\rm (iv)}]{Every ordered pair $(\al, \be)$ of orthogonal real roots of $W$ is $W$-conjugate to its reversal,
    $(\be, \al)$.}
    \item[\text{\rm (v)}]{Two ordered pairs of orthogonal real roots $(\al_1, \be_1)$ and $(\al_2, \be_2)$ are $W$-conjugate
   if and only if the corresponding unordered pairs $\{\al_1, \be_1\}$ and $\{\al_2, \be_2\}$ 
    are $W$-conjugate.
    The number of $W$-orbits in each case is equal to the number connected components of $\hpqr$. This
    number is 3 if $W$ is of type $D_4=Y_{1,1,1}$, is 2 if $W$ is of type $D_n=Y_{1,1,n-3}$ for
$n>4$, and is 1 otherwise.} 
\end{itemize}
\end{prop}

\begin{proof}
Any real root is $W$-conjugate to a simple root by definition, and the simple roots are in the same $W$-orbit by Lemma 
\ref{lem:simprootorbits} because $\ypqr$ is simply laced and connected. It follows that $W$ acts transitively on the set
$\realroots$ of real roots, proving (i).

Let $\al_1$ be a simple root that maximally far from the branch point of $\Gamma$, and let $W_{\al_1}$ be its stabilizer
in $W$. By the previous paragraph, any ordered pair of orthogonal real roots, $(\al, \be)$, is $W$-conjugate to one of the form 
$(\al_1, \gamma)$. By Theorem \ref{thm:hpqr}, $\gamma$ is a real root for the a simply laced Weyl group $W_{\al_1}$. It follows that
there exists $w \in W_{\al_1}$ such that $w(\gamma)$ is a simple root in the root system of $W_{\al_1}$. By Theorem 
\ref{cbminimal}, we have $w\big( (\al_1, \gamma) \big) = (\al_1, \be')$, where $\al_1 \vee \be' \in \B$ is a canonical basis element.

The explicit description of $\B$ in Definition \ref{def:canbas} shows that either (a) $\be'$ is a simple root of $W$, or 
(b) $\be' = \theta_1$. In the second case, Lemma \ref{lem:branchnodes} implies that $(\al_1, \be')$ is $W$-conjugate to an ordered pair of
simple roots of $W$. This completes the proof of (ii).

Part (iii) follows from (ii), because if $\al_i$ and $\al_j$ are orthogonal simple roots, then $\al_i \vee \al_j$ is an element of $\B$.

To prove (iv), it suffices by (ii) to consider the case where $\al$ and $\be$ are both simple roots. By repeatedly using the identity
$s_is_j(\al_i) = \al_j$ when $i$ and $j$ are adjacent vertices of $\Gamma$, we may assume that there is a subgraph $i$--$j$--$k$ of
$\Gamma$ in which $\al = \al_i$ and $\be = \al_k$. Direct calculation now shows that $$
s_j s_i s_k s_j\big( (\al, \be) \big) = (\be, \al)
,$$ from which (iv) follows.

The first assertion of (v) follows from (iv), so it is enough to prove the second assertion for ordered pairs of roots.
We claim that there is a bijection between the set of $W_{\al_1}$-orbits of real roots of $W_{\al_1}$ and the set of $W$-orbits
of ordered orthogonal pairs of real roots of $W$, given by $$
\phi([\gamma]) = [(\al_1, \gamma)]
,$$ where $[\gamma]$ is the $W_{\al_1}$-orbit of $\gamma$, and $[(\al_1, \gamma)]$ is the $W$-orbit of the pair $(\al_1, \gamma)$.
The map $\phi$ is well-defined and injective because $W_{\al_1}$ is the stabilizer of $\al_1$, and $\phi$ is surjective because $W$
acts transitively on $\simproots$.

It follows from Lemma \ref{lem:simprootorbits}, applied to the simply laced Weyl group $W_{\al_1}$, that the orbits of real roots of 
$W_{\al_1}$ are in bijection with the connected components of $\hpqr$. The number of these connected
components is as claimed by Lemma \ref{lem:componentcount}. 
\end{proof}

The following basic result from linear algebra turns out to be very helpful.

\begin{lemma}\label{lem:components}
Let $V$ be a finite dimensional vector space, and let $\al_1$ and $\al_2$ be two linearly independent vectors
in $V$. If there exist $\be_1, \be_2 \in V$ such that $\al_1 \vee \al_2 = \be_1 \vee \be_2$, then the vectors
$\be_i$ agree with the vectors $\al_i$ up to changing the order and multiplication by nonzero scalars.
\end{lemma}

\begin{proof}
We extend $\{\al_1, \al_2\}$ to a basis ${\mathcal A} = \{\al_1, \ldots, \al_n\}$ of $V$. Let 
$\{\al_i \vee \al_j : 1 \leq i \leq j \leq n\}$ be the associated standard basis of $S^2(V)$, 
and consider the expansion of $\be_1 \vee \be_2 = \al_1 \vee \al_2$ in 
terms of this standard basis. Because the coefficient of $\al_k \vee \al_k$ in $\al_1 \vee \al_2$ is
zero for all $k$, it follows that the supports of $\be_1$ and $\be_2$ with respect to ${\mathcal A}$ 
are disjoint. 

In turn, it follows that if $\al_k \vee \al_l$ is in the support of $\be_1 \vee \be_2$, then either
$\al_k$ is in the support of $\be_1$ and $\al_l$ is in the support of $\be_2$, or vice versa,
but not both. By considering the coefficient of $\al_k \vee \al_l$ in $\al_1 \vee \al_2$, this
can only be possible if either $k=1$ and $l=2$, or $l=1$ and $k=2$. This implies that either $\be_1$ is
a nonzero scalar multiple of $\al_1$ and $\be_2$ is a nonzero scalar multiple of $\al_2$, or vice
versa, which completes the proof.
\end{proof}

\begin{defn}\label{def:components}
If $v$ is an element of $S^2(V)$ of the form $\al \vee \be$, then we call $\al$ and $\be$ the 
{\it components} of $v$.
By Lemma \ref{lem:components}, the components of $v \in S^2(V)$ are well defined up to order and
multiplication by nonzero scalars.
We will therefore say ``$\al$ is a component of $v$" to mean the same as ``some scalar 
multiple of $\al$ is a component of $v$".
We call a \tworoot\ {\it real} (respectively, {\it positive}) if its components can be taken to be real
(respectively, positive). 
\end{defn}

\begin{prop}\label{prop:fibres}
Let $f$ be the function from the set of unordered pairs of orthogonal real roots of $\ypqr$ to
$\realtworoots$ defined by $$
f(\{\al, \be\}) = \al \vee \be
.$$ 
\begin{itemize}
    \item[\text{\rm (i)}]{The fibre of each \tworoot consists of the two pairs $\{\al, \be\}$ and $\{-\al, -\be\}$,
and these two pairs are conjugate to each other under the action of the Weyl group.}    
    \item[\text{\rm (ii)}]{The function $f$ induces a bijection between $W$-orbits of unordered
    pairs of orthogonal real roots, and $W$-orbits of real \tworoots.}
\end{itemize}

\end{prop}

\begin{proof}
The statement about fibres follows from Lemma \ref{lem:components} and the fact
(\cite[Proposition 5.1 (b)]{kac90}) that the only
scalar multiples of a real root $\al$ are $\pm \al$. The two pairs listed are conjugate to
each other by the Weyl group element $s_{\al} s_{\be}$, which completes the proof of (i).

For part (ii), Proposition \ref{orthorbits} (v) gives the equivalence between ordered and
unordered pairs of real roots. Part (i) then implies that the function $f$ gives a well-defined 
correspondence between ordered pairs of orthogonal real roots and real \tworoots.
\end{proof}

\begin{rmk}\label{rmk:thatremark}
Proposition \ref{prop:fibres} allows us to identify the action of $W$ on pairs of orthogonal real roots with the 
action of $W$ on real \tworoots. We will use this implicitly from now on, for example in 
Proposition \ref{prop:explicitorbits} below.
\end{rmk}

\begin{notation}\label{not:aandd}
In order to give a precise description of the $W$-orbits of \tworoots, we
recall the standard constructions of root systems of types $A$ and $D$ as described in 
\cite[\S2]{humphreys90}. We endow $\R^n$ with the usual positive definite inner product and with an orthonormal basis 
$\{\ep_1, \ep_2, \ldots, \ep_n\}$.

In type $A_{n-1}$, the positive roots are $\{\ep_i - \ep_j : 1 \leq i < j \leq n\}$, the simple roots are 
$\{\ep_i - \ep_{i+1} : 1 \leq i < n\}$, and the highest root is $\ep_1 - \ep_n$. The Weyl group is isomorphic to the 
symmetric group $S_n$ and it acts on the basis elements $\ep_i$ by permutations. For $1 \leq i < n$, the simple
reflection $s_i$ corresponding to $\al_i := \ep_i - \ep_{i+1}$ acts as the transposition $(i, i+1)$.

In type $D_n$, the positive roots are $\{\ep_i \pm \ep_j : 1 \leq i < j \leq n\}$, the simple roots are $$
\{\al_i := \ep_i - \ep_{i+1} : 1 \leq i < n\} \cup \{\al_n := \ep_{n-1} + \ep_n\}
,$$ and the highest root is $\ep_1 + \ep_2$. The numbering scheme for the simple roots is shown in
Figure \ref{fig:D}. The Weyl group acts on the 
elements $\pm \ep_i$ by signed permutations. For $1 \leq i < n$, the simple
reflection $s_i$ corresponding to $\ep_i - \ep_{i+1}$ acts as the transposition $(i, i+1)$. The simple reflection
$s_n$ corresponding to $\ep_{n-1} + \ep_n$ acts as the signed permutation switching $\ep_{n-1}$ and $-\ep_n$,
and fixing $\ep_j$ for $j \not\in \{n-1, n\}$.
\end{notation}


\begin{figure}[ht!]
    \centering
\begin{tikzpicture}

            \node[main node] (1) {};
            \node[main node] (2) [right=1cm of 1] {};
            \node[main node] (3) [right=1.5cm of 2] {};
            \node[main node] (4) [right=1cm of 3] {};
            \node[main node] (5) [above right=0.7cm and 0.9cm of 4] {};
            \node[main node] (6) [below right=0.7cm and 0.9cm of 4] {};

            \node (11) [below=0.1cm of 1] {\small{$1$}};
            \node (22) [below=0.1cm of 2] {\small{$2$}};
            \node (33) [below=0.1cm of 3] {\small{$n-3$}};
            \node (44) [right=0.1cm of 4] {\small{$n-2$}};
            \node (55) [right=0.1cm of 5] {\small{$n-1$}};
            \node (66) [right=0.1cm of 6] {\small{$n$}};

            \path[draw]
            (1)--(2)
            (3)--(4)--(5)
            (4)--(6);

            \path[draw,dashed]
            (2)--(3);
\end{tikzpicture}
\caption{The Dynkin diagram of type $D_{n}\, (n\ge 4)$}
    \label{fig:D}
\end{figure}
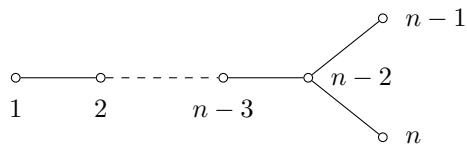

\begin{prop}\label{prop:explicitorbits}
Let $W$ be a simply laced Weyl group of finite type. 
\begin{itemize}
    \item[\text{\rm (i)}]{If $W$ is of type $A_n$ then there is a single orbit of positive \tworoots. The elements
    of $\B$ in this orbit are $$
\{ \al_i \vee \al_j : 1 \leq i < j-1 \leq n-1 \}
\quad \cup \quad
\{ \al_i \vee \eta_{i-1, i+1} : 1 < i < n \}
    .$$}
\item[\text{\rm (ii)}]{If $W$ is of type $D_4$ \textup{(}where $\theta_1=\theta_3=\theta_4$ is the
        highest root\textup{)}, then there are three orbits of positive \tworoots.
    The orbits intersect $\B$ in the sets $$
    \{ \al_1 \vee \al_3, \ \al_2 \vee \eta_{1,3}, \ \al_4 \vee \theta_4\}, \quad
    \{ \al_1 \vee \al_4, \ \al_2 \vee \eta_{1,4}, \ \al_3 \vee \theta_3\}, \quad \text{and} \quad 
    \{ \al_3 \vee \al_4, \ \al_2 \vee \eta_{3,4}, \ \al_1 \vee \theta_1\}
    .$$}
    \item[\text{\rm (iii)}]{If $W$ is of type $D_n$ for $n \geq 5$ then there are two orbits of positive \tworoots, $\trorbit_1$ and
    $\trorbit_2$, where $$
\trorbit_1 = 
\big\{
    (\ep_i - \ep_j)\vee(\ep_i + \ep_j) : 1 \leq i < j \leq n
\big\}
    $$ and $\trorbit_2 = \posroots^2 \backslash \trorbit_1$ where $\posroots^2$ is the set of positive
    2-roots. The elements of $\B$ in the orbit $\trorbit_1$ are the $n-1$
    elements $$
\al_1 \vee \theta_1, \ \al_2 \vee \theta_2, \ \ldots,\ \al_{n-3} \vee \theta_{n-3},\ \al_{n-2} \vee \eta_{n-1, n},\ 
\al_{n-1} \vee \al_n
    .$$}
    \item[\text{\rm (iv)}]{If $W$ is of type $E_6$, $E_7$, or $E_8$, then there is a single orbit of positive \tworoots.}
\end{itemize}
\end{prop}

\begin{proof}
In type $A$, two positive roots $\ep_i - \ep_j$ and $\ep_k - \ep_l$ are orthogonal if and only if their supports, $\{i, j\}$ 
and $\{k, l\}$, are disjoint. There is therefore a single $A_n$-orbit of orthogonal roots, which proves the first assertion for type $A_n$.
The second assertion follows by restricting Definition \ref{def:canbas} to a parabolic subgroup of type $A$.

Proposition \ref{orthorbits} (v) implies that there are three orbits in part (ii). The other claims of (ii) follow from computations
similar to those in Lemma \ref{lem:branchnodes}.

In type $D_n$ for $n \geq 5$, Proposition \ref{orthorbits} (v) shows that there are two orbits of positive \tworoots.
In this case, two positive roots are orthogonal if and only if their supports are either identical or disjoint.
Because the Weyl group acts by signed permutations, these two types of orthogonal roots must form separate orbits. This proves the 
statement describing $\trorbit_1$ and $\trorbit_2$.

To prove the last assertion of (iii), we need to identify all the root pairs in $\trorbit_1$ that contain a simple root.
Recall that $\theta_1$ is the highest root in type $D_n$ and that $\theta_1 = \ep_1 + \ep_2$.
It follows that $\{\al_1, \theta_1\} = \{\ep_1 - \ep_2, \ep_1 + \ep_2\}$, which is an element of $\trorbit_1$. Similarly, we have
$\theta_k = \ep_k + \ep_{k+1}$ for all $1 \leq k \leq n-3$. Direct calculation shows that $\eta_{n-1, n} = \ep_{n-2} + \ep_{n-1}$,
and the proof follows.

Part (iv) holds by Proposition \ref{orthorbits} (v). 
\end{proof}

\begin{defn}\label{def:largesmall}
If $W$ has type $D_n$ for $n \geq 5$, we will refer to the orbits $\trorbit_1$ and $\trorbit_2$ of 
Proposition \ref{prop:explicitorbits} (iii) as the {\it small orbit} and the {\it large orbit} of $W$, respectively.
\end{defn}

\begin{rmk}\label{rmk:smallorbit}
When $W$ has type $D_n$ for $n \geq 5$, the small orbit of \tworoots behaves like the root system of type $A_{n-1}$.
More precisely, a \tworoot of the form $(\varepsilon_i - \varepsilon_j) \vee (\varepsilon_i + \varepsilon_j)$, which can be simplified to 
$(\varepsilon_i \vee \varepsilon_i) - (\varepsilon_j \vee \varepsilon_j)$, can be identified with the root $\varepsilon_i - \varepsilon_j$ of type $A_{n-1}$. With this identification, 
the action of $W(D_n)$ by signed permutations is equivalent to the action of $W(A_{n-1})$ by unsigned permutations. 
This means that the action factors through the surjective homomorphism of groups from $W(D_n)$ to $W(A_{n-1})$ that sends 
the generators $s_{n-1}$ and $s_n$ of $W(D_n)$ to the same generator, $s_{n-1}$, of $W(A_{n-1})$.

In type $D_4$, the argument of the previous paragraph applies verbatim to the orbit $$
\{ \al_3 \vee \al_4, \ \al_2 \vee \eta_{3,4}, \ \al_1 \vee \theta_1\} =
\{ (\ep_3 - \ep_4) \vee (\ep_3 + \ep_4), \ (\ep_2 - \ep_3) \vee (\ep_2 + \ep_3), \ (\ep_1 - \ep_2) \vee (\ep_1 + \ep_2) \}
,$$ and it applies to the other two orbits of \tworoots by applying graph automorphisms. The action of $W$ on each of the
three orbits of \tworoots factors through a surjective homomorphism from $W(D_4)$ to $W(A_3) \cong S_4$ that identifies 
two of the three branch nodes.
\end{rmk}

\section{Reflections acting on \tworoots}\label{sec:reflact}

The goal of Section \ref{sec:reflact} is to prove Theorem \ref{thm:canbasrefl}, which gives a formula 
for the action of a simple reflection on a canonical basis element. 

\begin{defn}
For each real root $\al$, we define the element $C_\al$ of the group algebra $FW$ to be $s_\al - 1$.
\end{defn}

\begin{lemma}\label{lem:caction}
Let $\al$ be a real root of type $\ypqr$, let $s_\al$ be the associated reflection, and let 
$V$ be the reflection representation.
\begin{itemize}
    \item[{\rm (i)}]{If $v \in V$, then we have $C_\al(v) = -B(\al, v)\al$, and $C_\al$ acts on
    $V \otimes V$ as $$
    C_\al \otimes C_\al + C_\al \otimes 1 + 1 \otimes C_\al
    .$$}
    \item[{\rm (ii)}]{If $v \in S^2(V)$, then $C_\al(v)$ is of the form $\al \vee v'$ for
    some $v' \in V$.}
    \item[{\rm (iii)}]{If $\be$ is a real root orthogonal to $\al$, then we have $C_\al C_\be = 
    C_\be C_\al$. The product $C_\al C_\be$ acts as zero on $V$, and acts on $V \otimes V$ as $$
C_\al \otimes C_\be + C_\be \otimes C_\al
    .$$}
\end{itemize}
\end{lemma}

\begin{proof}
The formula for $C_\al(v)$ follows from the formula for the action of $s_\al$ on $V$.
Since $s_\al$ acts diagonally as $s_\al \otimes s_\al$ on $V \otimes V$, it follows that $C_\al = s_\al - 1$ acts
on $V \otimes V$ as $$
s_\al \otimes s_\al - 1 \otimes 1 = C_\al \otimes C_\al + C_\al \otimes 1 + 1 \otimes C_\al
,$$ which completes the proof of (i).

By part (i), it follows that if $v_1 \in V \otimes V$, then we have $$
C_\al(v_1) = \lambda_1 \al \otimes \al + \al \otimes v_2 + v_3 \otimes \al
$$ for some $v_2, v_3 \in V$. In particular, if $v_1 \in S^2(V)$ then we have $v_2 = v_3$ and $$
C_\al(v_1) = \al \vee (\lambda_1 \al + v_2),$$ which proves part (ii).

To prove (iii), note that $s_\al$ and $s_\be$ commute with each other because $\al$ and $\be$ are orthogonal.
It follows that $C_\al = s_\al - 1$ and $C_\be = s_\be - 1$ also commute with each other. If we take $v \in S^2(V)$,
part (i) implies that $C_\al C_\be (v)$ is a scalar multiple of $\al$, and that $C_\be C_\al(v)$ is a scalar
multiple of $\be$. It follows that $C_\al C_\be(v) = 0$. The formula for the action on $V \otimes V$ follows 
from this by composing the actions of $C_\al$ and $C_\be$ on $V \otimes V$ as given in (i).
\end{proof}

\begin{lemma}\label{lem:ctu}
Let $\al$ and $\be$ be real roots of type $\ypqr$ such that $B(\al, \be) = \pm 1$, and define
$t = s_\al$ and $u = s_\be$. Then we have $tut=utu$, and the element of $FW$ given by $$
C = 1 - t - u + tu + ut - tut
$$ acts as zero on the submodule $M \leq S^2(V)$.
\end{lemma}

\begin{proof}
Since $\al$ and $\be$ are real roots, we have $B(\al, \al) = 2 = B(\be, \be)$, so the condition $B(\al, \be) = \pm 1$
implies that $\al$ and $\be$ are linearly independent by linear algebra.

When computing products of $t$ and $u$, we note that since $s_\gamma=s_{-\gamma}$ for any real root
$\gamma$ of $W$, by replacing $\beta$ with $-\beta$ if necessary we may assume that $B(\alpha,\beta)=-1$. It then
follows from the formula for a reflection that $tut$ and $utu$ both negate the vector $\al + \be$
and fix its $B$-orthogonal complement, therefore $tut=utu = s_{\al +\be}$. 

Note that we have $$
C = (t-1)(-1+u-ut) = C_\al(-1+u-ut)
.$$ It follows from Lemma \ref{lem:caction} (i) that if $v \in S^2(V)$ then $C(v) = \al \vee v'$ for some $v' \in V$, 
so that $\al$ is a component of $C(v)$.

Similarly, we have $$
C = (u-1)(-1+t-tu) = C_\be(-1+t-tu) 
,$$ which implies that $\be$ is a component of $C(v)$.

Lemma \ref{lem:components} now implies that $C(v)$ is a scalar multiple of $\al \vee \be$. However, 
the hypothesis that $B(\al, \be) \ne 0$ implies that $\al \vee \be$ does not lie in $M$. We conclude that $C(v)=0$.
\end{proof}

\begin{rmk}\label{rmk:annihilation}
    In the setting of Lemma \ref{lem:ctu}, the element $C$ in fact annihilates all of the symmetric
    square $S^2(V)$ rather than just the module $M$. To see this, it suffices to show that
    $C\cdot (\alpha\vee\beta)=0$ because $M$ has codimension 1 in $S^2(V)$ and
    $\alpha\vee\beta\in S^2(V)\setminus M$. The fact that $C\cdot (\alpha\vee\beta)=0$ can be checked by a
    straightforward computation: assuming that $B(\alpha,\beta)=-1$ without loss of generality as in
    the proof of Lemma \ref{lem:ctu}, we have 
    \begin{eqnarray*}
        C\cdot (\alpha\vee\beta)&= & \alpha\vee \beta-  t\cdot(\alpha\vee\beta)-u\cdot(\alpha\vee
        \beta)+ tu\cdot(\alpha\vee \beta)+ut\cdot(\alpha\vee \beta)- tut\cdot(\alpha\vee \beta)\\
 &= & \alpha\vee\beta + \alpha \vee (\alpha+\beta) + (\alpha+\beta)\vee
 \beta-\beta\vee(\alpha+\beta)-(\alpha+\beta)\vee \alpha-\beta\vee\alpha\\
 &= &0.
    \end{eqnarray*}
    On the other hand, we note that $C$ does not annihilate all of the tensor square
    $V\otimes V$: a similar computation to the one shown above proves that 
if $\alpha,\beta$ and another root $\gamma$ are the simple roots of type $A_3$ subsystem of
$Y_{a,b,c}$ with $B(\alpha,\gamma)=0$ and $B(\beta,\gamma)=-1$, then $C\cdot (\beta\otimes
\gamma)=\alpha\otimes \beta-\beta\otimes \alpha\neq 0$. 
\end{rmk}

The next result describes a situation where applying a reflection to a \tworoot is analogous to applying a reflection 
to a root. In each case, one obtains a sum of two \tworoots: the original one, and a different \tworoot in the same $W$-orbit.

\begin{prop}\label{prop:generalrefl}
Let $\al \vee \be$ be an arbitrary real \tworoot of type $\ypqr$, and let $\gamma$ be a simple
root. 
\begin{itemize}
    \item[\text{\rm (i)}]{We have $s_\gamma(\al \vee \be) = (\al \vee \be) + (\gamma \vee v)$, where $v$ is given by $$
v = B(\al, \gamma) B(\be, \gamma) \gamma - B(\al, \gamma) \be - B(\be, \gamma) \al
.$$}
\item[\text{\rm (ii)}]{Let $\al \vee \be$ be a real \tworoot of type $\ypqr$, and let $\gamma$ be a real root for which $B(\al, \gamma) = \pm 1$.
Then we have $$
s_\gamma(\al \vee \be) = (\al \vee \be) + s_\al s_\gamma(\al \vee \be) = (\al \vee \be) \mp (\gamma \vee s_\al s_\gamma(\be))
.$$}
    \item[\text{\rm (iii)}]{The following are equivalent:
    \begin{itemize}
        \item[\text{\rm (1)}]{the element $v \in V$ from (i) is a real root;}
        \item[\text{\rm (2)}]{either $B(\al, \gamma) = \pm 1$, or $B(\be, \gamma) = \pm 1$, or both.}
    \end{itemize}}
\end{itemize}
\end{prop}

\begin{proof}
To prove (i), we use the formula for a reflection acting on $V$: \begin{align*}
 s_\gamma(\al \vee \be) &= (s_\gamma(\al) \vee s_\gamma(\be))\\
&= (\al - B(\gamma, \al) \gamma) \vee (\be - B(\gamma, \be) \gamma)\\
&= (\al \vee \be) - (B(\gamma, \al) \gamma \vee \be) - (\al \vee B(\gamma, \be) \gamma) + (B(\gamma, \al) \gamma \vee B(\gamma, \be) \gamma) \\
&= (\al \vee \be) + \big(\gamma \vee (-B(\gamma, \al) \be - B(\gamma, \be) \al + B(\gamma, \al)B(\gamma, \be) \gamma )\big),
\end{align*} and the stated formula follows.

To prove (ii), let $t$ and $u$ be the reflections associated to $\al$ and $\gamma$, respectively. Since $\al$ and $\be$ are orthogonal, it 
follows that $t$ fixes $\be$, so that we have $$
(t-1)(\al \vee \be) = (t(\al) \vee t(\be)) - (\al \vee \be) = -2(\al \vee \be)
.$$

Let $C$ be the element defined from $t, u$ in Lemma \ref{lem:ctu}, and note that we have $$
C = (-1+u-tu)(t-1)
.$$ By Lemma \ref{lem:ctu}, we have $C(\al \vee \be)=0$. Since $(t-1)(\al \vee \be)$ is a nonzero multiple of
$\al \vee \be$, it follows that we have $$
(-1+u-tu)(\al \vee \be)=0
,$$ which implies the first equation in the statement of (ii). The second equation follows because $s_\al s_\gamma(\al)$ is equal to
$\gamma$ if $B(\al, \gamma)=-1$ and to $-\gamma$ if $B(\al, \gamma) = +1$. This completes the proof of (ii).

To show (1) implies (2) in part (iii), assume that the element $v$ is a real root. It follows that we have $B(v, v) = 2$. For brevity, let
us define $x = B(\alpha, \gamma)$ and $y = B(\beta, \gamma)$, so that $v = xy\gamma - x\be - y\al$. We then have \begin{align*}
B(v, v) &= B(xy\gamma - x\be - y\al, xy\gamma - x\be - y\al)\\
&= x^2y^2 B(\gamma, \gamma) + x^2 B(\be, \be) + y^2 B(\al, \al) - 2x^2yB(\gamma, \be) -2xy^2B(\gamma, \al) + 2xyB(\be, \al) \\
&= 2x^2 y^2 + 2y^2 + 2x^2 - 2x^2y^2 - 2x^2y^2 + 0 \\
&= 2x^2 + 2y^2 - 2x^2 y^2.
\end{align*} Since we also know that $B(v, v)=2$, we have $2x^2 + 2y^2 - 2x^2y^2=2$. This is equivalent to the condition $$
(x^2 - 1)(y^2 - 1) = 0
,$$ so that either $x = \pm 1$ or $y = \pm 1$, as required.

Now assume that (2) holds. If $B(\al, \gamma) = \pm 1$, it follows from (ii) that 
$v = \mp s_\al s_\gamma(\be)$, 
which is a real root. The case $B(\be, \gamma) = \pm 1$ follows by a symmetrical argument, proving (1).
\end{proof}

\begin{rmk}\label{rmk:deltar}
If $\gamma$ is a simple root, then the root $s_\al s_\gamma(\be)$ in the statement of Proposition \ref{prop:generalrefl} (ii) 
agrees up to sign with the positive root $\delta_r$ appearing in \cite[Lemma 2.2]{cohen06}.
\end{rmk}

The next result gives a short formula for the action of a simple reflection on a canonical basis element in terms of
the element $v$ of Proposition \ref{prop:generalrefl} (i).

\begin{theorem}\label{thm:canbasrefl}
Let $\B$ be the canonical basis of \tworoots of type $\ypqr$, let $\al \vee \be \in \B$, and let $\gamma$ be a simple
root of $W$. Then we have $$
s_\gamma(\al \vee \be) = \begin{cases}
\al \vee \be & \text{\ if\ } B(\al, \gamma) = B(\be, \gamma) = 0;\\
-\al \vee \be & \text{\ if\ } \gamma \in \{\al, \be\};\\
(\al \vee \be) + (\gamma \vee v) & \text{\ otherwise,\ for\ some\ } (\gamma \vee v) \in \B.\\
\end{cases}
$$ Furthermore, the basis element $\gamma \vee v$ appearing above satisfies $w(\al \vee \be) = \gamma \vee v$ for
some $w \in \langle s_\al, s_\be, s_\gamma \rangle$.
\end{theorem}

\begin{proof}
By Theorem \ref{cbminimal} (ii), we may assume without loss of generality that $\al$ is a simple root.
If $\gamma\in \{\al, \be\}$ or $B(\al,\gamma)=B(\be,\gamma)=0$, then $s_\gamma(\al \vee \be)$ 
equals $-\al \vee \be$ or $\al \vee \be$ by direct computation. Otherwise, we must either have 
$B(\al, \gamma)=-1$ or simultaneously have $B(\al, \gamma)=0$ and $B(\be, \gamma)=-1$ by Lemma
\ref{abgcartan}. 
It remains to show that in both these cases, we have $s_\gamma(\al \vee \be)=(\al \vee \be)+(\gamma \vee v)$ for a 
canonical basis element $\gamma\vee v$ with the claimed properties. 


If we have $B(\al, \gamma) = -1$, then $\al$ and $\gamma$ correspond to adjacent vertices of $\Gamma$, 
and Proposition \ref{prop:generalrefl} (ii) implies that $$
s_\gamma(\al \vee \be) = (\al \vee \be) + s_\al s_\gamma(\al \vee \be)
.$$ Corollary \ref{cor:starop} now implies that $s_\al s_\gamma(\al \vee \be)$ is a canonical basis element of the
form $\gamma \vee v$.

The other possibility is that $B(\al, \gamma) = 0$ and $B(\be, \gamma) = -1$. In this case, 
Proposition \ref{prop:generalrefl} (ii) implies that $$
s_\gamma(\al \vee \be) = (\al \vee \be) + s_\be s_\gamma(\al \vee \be)
.$$ Since both $\be$ and $\gamma$ are orthogonal to $\al$, we have $s_\be s_\gamma(\al)=\al$.
The hypothesis $B(\be, \gamma) = -1$ implies that $s_\be s_\gamma(\be) = \gamma$. We conclude that $$
s_\gamma(\al \vee \be) = (\al \vee \be) + (\al \vee \gamma)
,$$ which completes the proof because we have $(\al \vee \gamma) \in \B$.
\end{proof}

We define the {\it \tworoot lattice} to be the $\Z$-span of the canonical basis $\B$. 

\begin{cor}\label{cor:tworootlat}
Let $W$ be the Weyl group of type $\ypqr$ and let $\B$ be the canonical basis of \tworoots of $M$.
\begin{itemize}
    \item[\text{\rm (i)}]{The action of $W$ 
    on the module $M$ leaves invariant the \tworoot lattice $\Z \B$.}
    \item[\text{\rm (ii)}]{Let $W_I$ be a parabolic subgroup of $W$, let $\allroots_I$ be the root system of
    $W_I$, and define  $$
\B_I = \{\al \vee \be\in \B : \al, \be \in \allroots_I\}
.$$ Then the action of $W_I$ leaves invariant the lattice $\Z \B_I$.}
\end{itemize}
\end{cor}

\begin{proof}
The formula in Theorem \ref{thm:canbasrefl} shows that a generator of $W$ sends a canonical basis element 
to an integral linear combination of canonical basis elements, which proves (i). 

To prove (ii), it is enough to show that if $s_i$ is a generator of $W_I$ and we have $\al \vee \be \in \B_I$, 
then $s_i(\al \vee \be) \in \Z \B_I$. This is immediate from Theorem \ref{thm:canbasrefl}, because the
\tworoots $\al \vee \be$ and $\gamma \vee v$ in that result are conjugate in $W_I$.
\end{proof}

\begin{rmk}\label{rmk:tworootlat}
Note that if the parabolic subgroup $W_I$ is also of type $\yg{a'}{b'}{c'}$ (for some values of $a'$, $b'$, and $c'$) then 
the set $\B_I$ coincides with the canonical basis $\B(a', b', c')$.
\end{rmk}

\section{Sign-coherence}\label{sec:coherence}

Following the theory of cluster algebras (\cite[Definition 2.2 (i)]{cao19}, 
\cite[Definition 6.12]{fomin07}), we say that a matrix $A$ is {\it column sign-coherent} (or
``sign-coherent" for short) if any two nonzero entries in the same column of $A$ have the same sign. 
We extend this terminology to say that a basis of a finite dimensional group representation $V$ is 
a {\it sign-coherent basis} of $V$ if every element of the group acts on $V$ by a sign-coherent 
matrix with respect to the basis, and we say $V$ is a {\it sign-coherent representation} if it admits 
a sign-coherent basis.


Sign-coherent representations exist in abundance. Some (trivial) examples of this phenomenon are representations
arising from permutations or signed permutations. An interesting and well-known example of a
sign-coherent basis is the basis of simple roots for the reflection representation of a Weyl group.
It also follows quickly from the definitions that a direct sum or tensor product of sign-coherent representations
is sign-coherent, as is the symmetric square of a sign-coherent representation. In particular, the standard basis
of $S^2(V)$ is a sign-coherent basis.

It is more difficult to find sign-coherent bases for irreducible modules, such as the direct summands of the
module $M$ in Theorem \ref{s2vdecomp} below. 
In this section, we will establish the following sign-coherence property of the canonical basis
$\B$:

\begin{theorem}\label{thm:coherence}
Let $W$ be a Weyl group of type $\ypqr$. The canonical basis $\B$ is a sign-coherent basis for the module $M$. With respect
to this basis, every element $w \in W$ is represented by a sign-coherent matrix of integers.
\end{theorem}

Because each real \tworoot is $W$-conjugate to an element of $\B$ (see Proposition \ref{orthorbits} (iii)), the \tworoots of $W$ are
precisely the set of possible columns of matrices representing the action of elements $w \in W$ with respect to the basis $\B$.
We can therefore restate Theorem \ref{thm:coherence} as follows:
\begin{theorem}\label{thm:coherence2}
Let $W$ be a Weyl group of type $\ypqr$. Then any real \tworoot $\al \vee \be$ of $W$ is an 
integral linear combination of elements of $\B$ with coefficients of like sign.
\end{theorem}
\noindent Here, the fact that any real \tworoot of $W$ is an integral linear combination of elements of
$\B(a,b,c)$ with coefficients of like sign is similar to the fact that any root of $W$ is an
integral linear combination of simple roots of $W$ with coefficients of like sign. 
Note also that Theorem \ref{thm:coherence2} implies that one can characterize the basis $\B$
as the set of positive \tworoots that cannot be expressed as a positive linear combination of other positive
\tworoots. 

\begin{rmk}
    In an earlier version of this paper, we conjectured that theorems \ref{thm:coherence} and
    \ref{thm:coherence2} hold for all Coxeter groups of type $\ypqr$ but proved the theorems only
    in the finite and affine cases (our conjecture beyond these types was based on extensive
    computer calculations). The proof for the general case that we will give below is based on a
    proof that was communicated to us by Robert B. Howlett.
\end{rmk}

To prove Theorem
\ref{thm:coherence2}, we note that Proposition \ref{orthorbits} (iii) and Corollary
    \ref{cor:tworootlat} (i) imply that any real \tworoot
$\al \vee \be$ is a linear combination of elements of $\B$ with integer coefficients, so it remains 
to prove that these integers are positive. We do so below. In the proof, we will freely use the result \cite[Proposition 5.7]{humphreys90} that if $w \in W$ and $\al$ is a positive real
root, then either $w(\al) > 0$ and $\ell(ws_\al) > \ell(w)$, or $w(\al) < 0$ and $\ell(ws_\al) < \ell(w)$.
We will also make use of the following remark in the proof. 

\begin{rmk}\label{rmk:posremark}
Because the components of any element of $\B$ can be taken to be positive roots, it follows that a
\tworoot is positive (respectively, negative) if and only it is a positive (respectively, negative)
linear combination of the standard basis of $S^2(V)$. In turn, this implies that if $w(\al \vee
\be)$ is a sign-coherent linear combination of elements of $\B$, then $w(\al \vee \be)$ is a
positive linear combination if $w(\al)$ and $w(\be)$ are both positive or both negative roots, and
$w(\al \vee \be)$ is a negative linear combination if one of $w(\al)$ and $w(\be)$ is a positive
root and the other is a negative root.
\end{rmk}

\begin{proof}[Proof of Theorem \ref{thm:coherence2}]
    By the discussions following Theorem \ref{thm:coherence}, to prove Theorem \ref{thm:coherence2} it
suffices to show that for any $w\in W$ and $\al\vee\beta\in \B$, the 2-root $w(\al\vee \beta)$ is a
linear combination of $\B$ with coefficients of like sign.  We prove this fact by induction
on the length, $\ell(w)$, of $w \in W$. The case $\ell(w) = 0$ is trivial, and the case $\ell(w) =
1$ follows from Theorem \ref{thm:canbasrefl}. Suppose then that we have $\ell(w) > 1$.

If we have $w(\al) < 0$, then we have $\ell(ws_\al) < \ell(w)$ and $w(\al \vee \be) = - ws_\al(\al
\vee \be)$, and the proof is completed by applying the inductive hypothesis to $ws_\al$. A similar
argument applies if $w(\be) < 0$, so we may assume from now on that both $w(\al) > 0$ and $w(\be) >
0$.

Fix a simple root $\gamma$ with the property that $\ell(w s_\gamma) < \ell(w)$, which implies that
$ws_\gamma(\gamma) > 0$. If we have $s_\gamma(\al \vee \be) = \pm (\al \vee \be)$, then the proof
follows by applying the inductive hypothesis to $ws_\gamma$ as in the previous paragraph. We may
therefore assume that we are in the third case of the statement of Theorem \ref{thm:canbasrefl}, so that $\gamma
\not\in \{\al, \be\}$, and $\gamma$ is not orthogonal to both $\al$ and $\be$. Since $\gamma$ and
$\al$ are distinct simple roots, we must have $B(\al, \gamma) \in \{0, -1\}$. 

Suppose that $ws_\gamma(\al) < 0$, which implies that $\ell(w s_\gamma s_\al) < \ell(w s_\gamma)$.
The assumption that $w(\al) > 0$ implies that $B(\al, \gamma) \ne 0$, and it follows from the
previous paragraph that $B(\al, \gamma) = -1$. Corollary \ref{cor:starop} implies that $s_\al s_\gamma(\al \vee \be)
= (\gamma \vee s_\al s_\gamma(\be))$ is an element of $\B$.  We then have $$ w(\al \vee \be) = w
s_\gamma s_\al(\gamma \vee s_\al s_\gamma(\be)) ,$$ and the proof follows by applying the inductive
hypothesis to $w s_\gamma s_\al$.

Suppose that $ws_\gamma(\be) < 0$, which implies that $\ell(w s_\gamma s_\be) < \ell(w s_\gamma)$.
Let $c:=B(\beta,\gamma)$. The assumption that $w(\be) > 0$ implies that $B(\be, \gamma) \ne 0$, and Lemma \ref{abgcartan} then
implies that $c = \pm 1$; in particular, we have $c^2=1$. It follows that
\[
    s_\be s_\gamma(\be) = s_\beta(\beta-c\gamma)=-\beta-c(\gamma-c\beta)=-c\gamma
    .
\]
We claim that one of  $\pm s_\beta s_\gamma(\alpha\vee\beta)$ is an element of $\B$. If $B(\al, \gamma) = 0$, then we
have $s_\be s_\gamma(\al) = \al$ and $s_\be s_\gamma(\al \vee \be) = \al \vee
(-c\gamma)=-c(\al\vee\gamma)$, which
proves the claim in this case because $\al$ and $\gamma$ are orthogonal simple roots and $c=\pm 1$. The other
possibility is that $B(\al, \gamma) = -1$, in which case we have \[s_\alpha
s_\gamma(\alpha)=\gamma\] and 
\[
    s_\beta s_\gamma(\alpha)=
    s_\beta(\alpha+\gamma)=\alpha+\gamma-c\beta=s_\alpha(\gamma-c\beta)=-cs_\alpha s_\gamma(\beta).
\]
Combining the three equations displayed above, we find that 
\[
s_\be s_\gamma(\al \vee \be) =
(-cs_\alpha s_\gamma(\beta))\vee (-c\gamma)= 
s_\alpha s_\gamma(\beta)\vee s_\alpha s_\gamma(\alpha)= 
    s_\alpha s_\gamma(\beta\vee\alpha)
\]
which completes the proof of the claim by Corollary \ref{cor:starop}.  The claim then
implies that $$ w(\al \vee \be) = w s_\gamma s_\be(s_\be s_\gamma(\al \vee \be)) ,$$ and the proof
follows by applying the inductive hypothesis to $w s_\gamma s_\be$.

By the previous three paragraphs, we may assume from now on that $ws_\gamma(\al)$, $ws_\gamma(\be)$,
and $ws_\gamma(\gamma)$ are all positive roots.

If we have $B(\al, \gamma) = 0$ then, since we are assuming that we are not in the first two cases
of the statement of Theorem \ref{thm:canbasrefl}, we have $B(\be, \gamma) = -1$ by  Lemma \ref{abgcartan}. This implies that
$s_\gamma(\al \vee \be) = (\al \vee \be) + (\al \vee \gamma)$, and we therefore have $$ w(\al \vee
\be) = ws_\gamma(s_\gamma(\al \vee \be)) = ws_\gamma(\al \vee \be) + ws_\gamma(\al \vee \gamma) .$$
It follows from the previous paragraph that $ws_\gamma(\al \vee \be)$ and $ws_\gamma(\al \vee
\gamma)$ are positive \tworoots.  Remark \ref{rmk:posremark} and the inductive hypothesis applied to
$ws_\gamma$ then imply that each of $ws_\gamma(\al \vee \be)$ and $ws_\gamma(\al \vee \gamma)$ is a
nonnegative integral linear combination of elements of $\B$. It follows that $w(\al \vee \be)$ is
also a nonnegative integral linear combination of elements of $\B$, which completes the proof in
this case.

We may suppose from now on that $B(\al, \gamma) = -1$. Let $c = B(\be, \gamma)$, so that $c \in
\{-1, 0, 1\}$ by Lemma \ref{abgcartan}.  Define $\be' = s_\al s_\gamma(\be) = \be - c(\al +
\gamma)$, and note that $s_\al(\be') = \be - c\gamma$ is also a root.  Suppose that $ws_\gamma(\be')
< 0$, which implies that $\ell(ws_\gamma s_{\be'}) < \ell(ws_\gamma)$.  The assumption that
$ws_\gamma(\be) > 0$ implies that $\be \ne \be'$, which rules out the case $c = 0$.  We now have
$c^2 = 1$, $s_{\be'} s_\gamma(\al) = \pm \be$ and $s_{\be'} s_\gamma(\be) = \pm \al$.  This implies
that $$ w(\al \vee \be) = ws_\gamma s_{\be'}(s_{\be'} s_\gamma(\al \vee \be)) = \pm ws_\gamma
s_{\be'}(\al \vee \be) ,$$ and the proof follows by applying the inductive hypothesis to $ws_\gamma
s_{\be'}$.

We have reduced to the case where $B(\al, \gamma) = -1$ and $w s_\gamma(\be') > 0$.  We have $s_\al
s_\gamma(\al) = \gamma$, and $s_\al s_\gamma(\be) = \be'$.  The case of $B(\al, \gamma) = -1$ in the
proof of Theorem \ref{thm:canbasrefl} now implies that $$ w(\al \vee \be) = ws_\gamma(s_\gamma(\al \vee \be)) =
ws_\gamma(\al \vee \be) + ws_\gamma(s_\al s_\gamma(\al \vee \be)) = ws_\gamma(\al \vee \be) +
ws_\gamma(\gamma \vee \be') ,$$ and Corollary \ref{cor:starop} implies that $(\gamma \vee \be') = s_\al s_\gamma(\al
\vee \be)$ is an element of $\B$.  We have shown that all of $ws_\gamma(\al)$, $ws_\gamma(\be)$,
$ws_\gamma(\gamma)$, and $w s_\gamma(\be')$ are positive roots, which implies that $ws_\gamma(\al
\vee \be)$ and $ws_\gamma(\gamma \vee \be')$ are positive \tworoots. Remark \ref{rmk:posremark} and
the inductive hypothesis applied to $ws_\gamma$ then imply that each of $ws_\gamma(\al \vee \be)$
and $ws_\gamma(\gamma \vee \be')$ is a nonnegative integral linear combination of elements of $\B$.
It follows that $w(\al \vee \be)$ is also a nonnegative integral linear combination of elements of
$\B$, which completes the proof.
\end{proof}

In finite types types $A$ and $D$, Theorem \ref{thm:coherence2} can be interpreted diagrammatically using the conventions
of Notation \ref{not:aandd}. 

We can depict positive roots of types $A$ and $D$ as arcs connecting rows of dots labelled $1, 2, \ldots, n$.
A positive root of the form $\ep_i - \ep_j$ (respectively, $\ep_i + \ep_j$) is depicted as an undecorated
(respectively, decorated) arc joining point $i$ to point $j$. We can then depict positive \tworoots as
pairs of (possibly decorated) arcs connecting points $i$ and $j$. In type $D$, this may result in two arcs 
connecting the same two points, one of which is decorated and one of which is not.

In this context, the linear relations between \tworoots that one obtains from Theorem
\ref{thm:coherence2} can be 
interpreted as a type of skein relation with positive
coefficients. For example, in type $A$ the positive \tworoot $(\al_1 + \al_2) \vee (\al_2 + \al_3)$ decomposes 
into a positive linear combination of canonical basis elements by Theorem \ref{thm:coherence2}: $$
(\al_1 + \al_2) \vee (\al_2 + \al_3) = (\al_1 \vee \al_3) + (\al_2 \vee (\al_1 + \al_2 + \al_3))
.$$ Writing this in terms of coordinates, we have \begin{equation}\label{eq:skein1}
(\ep_1 - \ep_3) \vee (\ep_2 - \ep_4) = 
\big( (\ep_1 - \ep_2) \vee (\ep_3 - \ep_4) \big)
+
\big( (\ep_2 - \ep_3) \vee (\ep_1 - \ep_4) \big) 
.
\end{equation} Pictorially, this shows how to express a diagram with a crossing as a positive linear combination of diagrams
with fewer crossings. The canonical basis elements in this case correspond to the legal configurations of arcs in
the top half of diagrams for the Temperley--Lieb algebra.


\begin{figure}[ht!]
    \centering
    \begin{tikzpicture}[anchorbase]      
        \draw (0,0) arc(-180:0:0.75) to (1,0);
        \draw (0.75,0) arc(-180:0:0.75) to (2.25,0);
        \draw (-0.3,0)--(2.55,0);
    \end{tikzpicture}
    \ \;$\longrightarrow$\;\
    \begin{tikzpicture}[anchorbase]
        \draw (0,0) arc(-180:0:0.375) to (0.75,0);
        \draw (1.5,0) arc(-180:0:0.375) to (2.25,0);
        \draw (-0.3,0)--(2.55,0);
    \end{tikzpicture}
    \ \;$+$\;\
    \begin{tikzpicture}[anchorbase]
        \draw (0,0) arc(-180:0:1.125) to (2.25,0);
        \draw (0.75,0) arc(-180:0:0.375) to (1.5,0);
        \draw (-0.3,0)--(2.55,0);
    \end{tikzpicture}
    \caption{Equation \ref{eq:skein1} interpreted as a skein relation}
\label{fig:skein1}
    \end{figure}
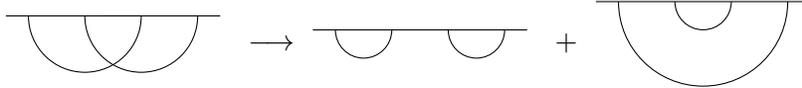

Something similar happens in type $D$. The positive \tworoot $(\ep_1 + \ep_4) \vee (\ep_2 + \ep_3)$ in type $D_4$
decomposes into the following positive linear combination of canonical basis elements: \begin{align}\label{eq:skein2}
(\ep_1 + \ep_4) \vee (\ep_2 + \ep_3) 
&= (\al_1 \vee \al_3) + (\al_2 \vee \eta_{1,3}) + (\al_4 \vee \theta_4) \\
&= 
\big( (\ep_1 - \ep_2) \vee (\ep_3 - \ep_4) \big) +
\big( (\ep_2 - \ep_3) \vee (\ep_1 - \ep_4) \big) +
\big( (\ep_3 + \ep_4) \vee (\ep_1 + \ep_2) \big). \nonumber
\end{align} Pictorially, this shows how to express a diagram with a non-exposed decorated arc (in this case, the
one between 2 and 3) as a linear combination of diagrams that have fewer such features. After performing a left-right
reflection, the canonical basis elements in this case correspond to the legal configurations of arcs in the top half of 
diagrams for the Temperley--Lieb algebra of type $D$, as described by the first author in \cite[Theorem 4.2]{green98}. 
The positive \tworoots of the form $(\ep_i - \ep_j) \vee (\ep_i + \ep_j)$ correspond to the ``diagrams of type 1" of
\cite{green98}, and the other positive \tworoots correspond to the ``diagrams of type 2".


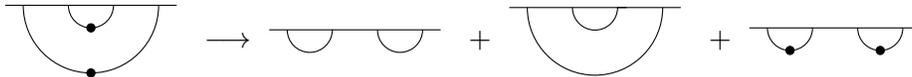
\begin{figure}[ht!]
    \centering
    \begin{tikzpicture}[anchorbase]      
        \draw (0,0) arc(-180:0:0.75*1.2) to (1.5*1.2,0);
        \draw (0.5*1.2,0) arc(-180:0:0.25*1.2) to (1*1.2,0);
        \draw (-0.2*1.2,0)--(1.7*1.2,0);
        \node[small node] at (0.75*1.2,-0.25*1.2) (a) {};
        \node[small node] at (0.75*1.2,-0.75*1.2) (b) {};
    \end{tikzpicture}
    \ \;$\longrightarrow$\;\
    \begin{tikzpicture}[anchorbase]
        \draw (0,0) arc(-180:0:0.25*1.2) to (0.5*1.2,0);
        \draw (1*1.2,0) arc(-180:0:0.25*1.2) to (1.5*1.2,0);
        \draw (-0.2*1.2,0)--(1.7*1.2,0);
    \end{tikzpicture}
    \ \;$+$\;\
    \begin{tikzpicture}[anchorbase]
        \draw (0,0) arc(-180:0:0.75*1.2) to (1.5*1.2,0);
        \draw (0.5*1.2,0) arc(-180:0:0.25*1.2) to (1.1*1.2,0);
        \draw (-0.2*1.2,0)--(1.7*1.2,0);
    \end{tikzpicture}
    \ \;$+$\;\
    \begin{tikzpicture}[anchorbase]
        \draw (0,0) arc(-180:0:0.25*1.2) to (0.5*1.2,0);
        \draw (1*1.2,0) arc(-180:0:0.25*1.2) to (1.5*1.2,0);
        \draw (-0.2*1.2,0)--(1.7*1.2,0);
        \node[small node] at (0.25*1.2,-0.25*1.2) (a) {};
        \node[small node] at (1.25*1.2,-0.25*1.2) (b) {};
    \end{tikzpicture}
    \caption{Equation \ref{eq:skein2} interpreted as a skein relation}
\label{fig:skein2}
    \end{figure}

From this point of view, the decoration rules for arcs in these algebras are canonically determined
by the basis $\B$.

\section{The highest \tworoot}\label{sec:highest}

In Section \ref{sec:highest}, we assume that $\Gamma$ is a Dynkin diagram of
finite type $A_n$, $D_n$, $E_6$, $E_7$, or $E_8$ unless otherwise stated, and
we continue to work over a subfield of $\R$.

Recall that the root lattice $Q=\Z \simproots$ of arbitrary type is equipped
with a standard partial order: if $q_1, q_2 \in \Z \simproots$, we say that
$q_1 < q_2$ if $q_2 - q_1$ is a positive linear combination of simple roots
$\simproots$.  Also recall that the {\it height} of a positive root
$\alpha=\sum_{b \in \simproots}\lambda_b b$ is defined to be the number
$\rootht(\alpha):=\sum_{b\in \simproots} \lambda_b$.

The basis $\B$ allows us to define natural analogues of $\le$ and $\rootht$ for
\tworoots: for \tworoots $q_1, q_2 \in \Z \B$, we say that $q_1 \leq_2 q_2$ if
$q_2 - q_1$ is a positive linear combination of elements of $\B$.  If $\al
\vee \be$ is a positive \tworoot satisfying $\al \vee \be = \sum_{b \in \B}
\lambda_b b$, then we may define the {\it height} of $\al \vee \be$ to be the
number $\tworootht(\alpha\vee\beta):=\sum_{b \in \B} \lambda_b$.  Note that in
this context, Theorem \ref{thm:coherence2} implies that every real \tworoot is
comparable to the zero vector in the order $\leq_2$. The same theorem also
implies that the height of a positive root must be a positive integer.
Theorem \ref{thm:canbasrefl} shows that if $s_i$ is a generator for $W$, $b
\in \B$ is a canonical basis element, and $s_i(b)$ is not a scalar multiple of
$b$, then $b <_2 s_i(b)$ is a covering pair.

The main purpose of this section is to show that each $W$-orbit of \tworoots in
$\Z\mathcal{B}$ has a unique maximal element with respect to the order
$\le_2$. To this end, we first use $\le_2$ to induce a partial order on the
set of pairs of orthogonal positive roots of $W$, also denoted by $\le_2$,
defined by $$ \{\al_1, \be_1\} \leq_2 \{\al_2, \be_2\} \text{\rm \ if } (\al_1
\vee \be_1) \leq_2 (\al_2 \vee \be_2).$$ The new order is well-defined since
$\alpha\vee\beta=\beta\vee\alpha$ for all roots $\alpha,\beta$ of $W$.
Similarly, we may define the {\it height} of each pair of positive orthogonal
roots $\{\alpha,\beta\}$ to be $\tworootht(\alpha\vee\beta)$.

To establish the existence of maximal \tworoots in $\Z\mathcal{B}$, we shall
compare the order $\le_2$ on root pairs with two other partial orders on sets
of orthogonal roots introduced in \cite{cohen06}. Motivated by the
Lawrence--Krammer representation of the Artin group, Cohen, Gijsbers, and Wales
defined combinatorially in \cite{cohen06} a partial order $\le'$ on each
$W$-orbit of $k$-tuples of mutually orthogonal positive roots, where $W$ is a
simply laced finite Weyl group and $k\in \Z_{>0}$.  Here, the action of each
element $w\in W$ sends every $k$-tuple $\rho=\{\alpha_1,\dots,\alpha_k\}$ to
the set \[w(\ome)=\posroots\cap \{\pm\alpha_1,\dots,\pm\alpha_k\}\] where
$\posroots$ is the set of positive roots of $W$. The definition of the order
$\le'$ also requires the $k$-tuples to be ``admissible", a technical
combinatorial property that is always satisfied if $k=2$ (see
\cite[Proposition 2.3]{cohen06}). As a consequence, the order $\le'$ restricts
to pairs of orthogonal roots as follows.

\begin{defn}\label{def:cgworder} Let $W$ be a simply laced Weyl group of
finite type. Let $\ome_1$ and $\ome_2$ be two (unordered) pairs of
orthogonal positive roots of $W$ such that $w(\ome_1) = \ome_2$ for some
$w \in W$. We say that $\ome_1 <' \ome_2$ if there exist $\gamma_1 \in
\ome_1 \backslash \ome_2$ and $\gamma_2 \in \ome_2 \backslash \ome_1$,
of minimal height in $\ome_1 \backslash \ome_2$ and $\ome_2 \backslash
\ome_1$ respectively, such that $\rootht(\gamma_1) < \rootht(\gamma_2)$.
\end{defn}

\begin{prop}[Cohen, Gijsbers, Wales] \label{prop:monoidal} Let $W$ be a simply
laced Weyl group of finite type, and let $\trorbit$ be a $W$-orbit of
pairs of orthogonal positive roots. The relation $\leq'$ of Definition
\ref{def:cgworder} is a partial order on $\trorbit$.
\end{prop}

\begin{proof} This follows from the proof of \cite[Proposition 3.1]{cohen06}.
(We note that the proof from \cite{cohen06} asserts that ``it is readily
verified" that $\leq'$ is an ordering. However, to our knowledge it is
not completely trivial to prove the fact that $\leq'$ is transitive.)
\end{proof}

The second partial order we need from \cite{cohen06} is defined using $\le'$
as follows.

\begin{defn}\label{def:monoidal} Let $W$ be a simply laced Weyl group of
finite type, and let $\trorbit$ be a $W$-orbit of pairs of orthogonal
positive roots. Let $\leq_m$ be the partial order on $\trorbit$ whose
covering relations are those of the form $\ome <_m s_i(\ome)$ where
$s_i$ is a Coxeter generator such that $s_i(\ome) \in \trorbit
\backslash \{\ome\}$ and $\ome <' s_i(\ome)$. We call $\leq_m$ the {\it
monoidal order} on $\trorbit$.
\end{defn}

\noindent Note that because $\leq'$ is a partial order, it follows that
$\leq_m$ is antisymmetric, and thus that the reflexive, transitive
extension of the relation in Definition \ref{def:monoidal} is a
partial order. It is immediate from the definitions that $\leq'$ is a
refinement of $\leq_m$.

The next result also applies in type $A$, by using the identifications of 
Corollary \ref{cor:tworootlat} (ii).

\begin{prop}\label{prop:refinement}
Let $W$ be a simply laced Weyl group of finite type, and let $\trorbit$ be a $W$-orbit of pairs of orthogonal
positive roots. 
\begin{itemize}
\item[\rm (i)]{If $\ome_1 = \{\al, \be\} \in \trorbit$ and $\ome_1 <_m \ome_2$ is a covering pair in $\trorbit$, then we have
$\ome_2 = \{s_i(\al), s_i(\be)\}$ for some simple reflection $s_i$. Furthermore, if 
$x = B(\al_i, \al)$ and $y = B(\al_i, \be)$, then we have $x, y \in \{-1, 0, 1\}$, and we do not have
$x = y = 0$.}
\item[\rm (ii)]{The partial order $\leq_2$ refines the monoidal order $\leq_m$; in other words, 
if $\ome_1, \ome_2 \in \trorbit$ satisfy $\ome_1 \leq_m \ome_2$, then we have $\ome_1 \leq_2 \ome_2$.}
\end{itemize}
\end{prop}

\begin{proof}
To prove (i), we note that $\ome_1=\{\alpha,\beta\} \subseteq \posroots$.  Then by Definition
\ref{def:monoidal} we must have $$ \ome_1<'\ome_2=s_i(\ome_1)=\posroots\cap \{\pm s_i(\alpha),\pm
s_i(\beta)\} $$ for some Coxeter generator $s_i$.  Let  $\alpha_i\in \Pi$ be the simple root
corresponding to $s_i$, and let $x=B(\alpha_i,\alpha)$, and let $y=B(\alpha_i,\beta)$ as in the
statement. Note that if $\alpha_i=\alpha$ then we have $s_i(\alpha)=-\alpha,s_i(\beta)=\beta$ and
$\ome_2=\{\alpha,\beta\}$.  This contradicts the fact that $\ome_1\neq \ome_2$, which proves that
$\alpha_i$ and $\alpha$ are distinct real roots.  Since $s_i$ permutes the set
$\posroots\setminus\{\alpha_i\}$, it follows that  $s_i(\alpha)\in \posroots$, A similar argument
shows that $\beta\neq \alpha_i$ and $s_i(\beta)\in \Phi_+$, so it follows that
$\rho_2=\{s_2(\alpha),s_i(\beta)\}$.

To prove the claims about $x$ and $y$, note first $\alpha\neq -\al_i$ since both $\al$ and
$\alpha_i$ are positive roots.  Also note that since $W$ is simply laced, the roots $\alpha_i$ and
$\alpha$ have the same norm in the sense that we must have
$B(\alpha_i,\alpha_i)=B(\alpha,\alpha)=2$. Since $\alpha_i$ and $\alpha$ are distinct, are not
opposite, and have the same norm, it then follows from \cite[\S 9.4]{Humphreys78} that
$x=B(\alpha_i,\alpha)\in \{-1,0,1\}$.  Similarly, we have $y\in\{-1,0,1\}$.  Moreover, since
$\ome_2\neq \ome_1$, the root $\alpha_i$ cannot be orthogonal to both $\alpha$ and $\beta$, so we
cannot have $x=y=0$. This completes the proof of (i).

It is enough to prove (ii) in the case where $\ome_1<_m \ome_2$ forms a covering pair. By the
previous paragraph, we may assume without loss of generality that $x\in \{-1,1\}$, and that the
\tworoot corresponding to $\ome_2$ is \[ s_i(\alpha)\vee s_i(\beta)=s_i(\alpha\vee\beta)=
    (\alpha\vee\beta)+s_\alpha s_{i}(\alpha\vee\beta) = \alpha\vee\beta+s_\alpha s_i(\alpha)\vee
s_\alpha s_i(\beta), \] where the second equality holds by Proposition \ref{prop:generalrefl} (ii)
since $x \in \{-1,1\}$. Every positive \tworoot is a linear combination of elements of $\B$ with
positive integral coefficients by Theorem \ref{thm:coherence2}, so to prove $\ome_1\le_2\ome_2$ it
now suffices to show that $s_\alpha s_i(\alpha)\vee s_\alpha s_i(\beta)$ equals a positive root.  We
do so by showing that $s_\alpha s_i(\alpha)$ and $s_\alpha s_i(\beta)$ are either both positive or
both negative roots, depending on the values of $x$ and $y$. 

Suppose first that $y = 0$. In this case, both $\al_i$ and $\al$ are orthogonal to $\be$, which
implies that $s_\al s_i(\be) = \be$ is positive. By assumption, we have $$ \{\al, \be\} <_m
\{s_i(\al), s_i(\be)\} = \{\al - x\al_i, \be\} ,$$ which implies (using Definition
\ref{def:cgworder}) that we have $x = -1$. In turn, this implies that $s_\al s_i(\al) = \al_i$ and
$s_\al s_i(\be) = \be$ are both positive roots, which completes the proof in this case.

Next, suppose that $y = x = \pm 1$. In this case, the condition that $\{\al, \be\} <_m \{s_i(\al),
s_i(\be)\}$ implies that $y=x=-1$.  This implies that $s_\al s_i(\al) = \al_i$ and $s_\al s_i(\be) =
\al + \be + \al_i$ are both positive, as required.

Finally, suppose that $y = -x = \pm 1$. We cannot have $\rootht(\al) = \rootht(\be)$, because one of
$s_i(\al)$ and $s_i(\be)$ would have a lower height than both of $\al$ and $\be$, which is
incompatible with the condition that $\{\al, \be\} <_m \{s_i(\al), s_i(\be)\}$. We may therefore
assume without loss of generality that $\rootht(\al) < \rootht(\be)$. The condition $\{\al, \be\}
<_m \{s_i(\al), s_i(\be)\}$ then implies that $x = -1$ and $y = 1$. This implies that $s_\al
s_i(\al) = \al_i$, a positive root, and $s_\al s_i(\be) = \be - \al - \al_i$. Because $\be - \al -
\al_i$ is a root, it cannot have height zero, so we must have $\rootht(\be) \geq \rootht(\al) + 2$
and thus that $\rootht(\be - \al - \al_i) > 0$. It follows that $\be - \al - \al_i$ is a positive
root, as required, which completes the proof of (ii).
\end{proof}

The next example shows that the partial order $\leq_2$ strictly refines the order $\leq_m$.

\begin{exa}\label{exa:refinement}
Let $\Gamma$ be a Dynkin diagram of type $D_5$, with vertices numbered $1$--$2$--$3$--$4$ and $3$--$5$,
so that $3$ is the branch point. The pair of positive orthogonal roots $\{\al_3, \theta_1\}$ 
is not minimal in $\leq_2$, because we have $$
\al_3 \vee \theta_1 = (\al_3 \vee \al_1) + (\al_3 \vee \eta_{2,4}) + (\al_3 \vee \eta_{2,5})
.$$ However, $\{\al_3, \theta_1\}$ is minimal in the order $\leq_m$, because there is no simple reflection
$s_i$ for which $s_i(\{\al_3, \theta_1\}) <_m \{\al_3, \theta_1\}$. 
\end{exa}

\begin{theorem}\label{thm:highest}
Let $W$ be a simply laced Weyl group of finite type, and let $\trorbit$ be a $W$-orbit of pairs of orthogonal
positive roots. 
The orbit contains a maximum element $\{\alpha,\beta\}$ with
respect to $\le_2$. In particular, if $\alpha\vee\beta=\sum_{b\in \B}\mu_b b$
and $\alpha' \vee \beta'=\sum_{b\in \B}\lambda_b b$ for some other element
$\{\alpha', \beta'\}$ in the orbit, then we have $\lambda_b\le \mu_b$ for all
$b\in \B$.
\end{theorem}

\begin{proof}
The orbit $\trorbit$ has a unique maximal element with respect to $\leq_m$ by \cite[Corollary 3.6]{cohen06},
and this is equivalent to having a maximum element because $\trorbit$ is finite. The result now follows from
Proposition \ref{prop:refinement} (ii).
\end{proof}

Recall from Proposition \ref{prop:fibres} and Remark \ref{rmk:thatremark} that the
$W$-orbits of pairs of orthogonal real roots and the $W$-orbits of real
\tworoots can be identified under the correspondence $\{\al', \be'\} \leftrightarrow
\al' \vee \be'$. Theorem \ref{thm:highest} may be interpreted as saying that the
identification matches the maximum elements of these orbits. 

In the sequel, we will refer to the \tworoot $\alpha\vee\beta$ from Theorem \ref{thm:highest} the {\it highest
\tworoot} in its $W$-orbit. Our next goal is to give an explicit description of the highest \tworoot in
each orbit, and the next result will be helpful for this purpose.

\begin{lemma}\label{lem:findhighest}
Let $W$ be a simply laced Weyl group of finite type, and let $\{\al, \be\}$ be a pair of orthogonal positive
roots of $W$ satisfying $\rootht(\al) = \rootht(\be)$. Suppose that every simple root $\al_i$ satisfies the 
following two conditions:
\begin{itemize}
    \item[\text{\rm (i)}]{if $B(\al_i, \al) = -1$ then $B(\al_i, \be) = +1$;}
    \item[\text{\rm (ii)}]{if $B(\al_i, \be) = -1$ then $B(\al_i, \al) = +1$.}
\end{itemize}
Then $\al \vee \be$ is the highest \tworoot in its orbit.
\end{lemma}

\begin{proof}
Suppose that the conditions are satisfied, but that $\{\al, \be\}$ is not the highest \tworoot with respect
to $\leq_m$. By Proposition \ref{prop:refinement} (i), there must be a simple reflection $s_i$ such that
$\al \vee \be \leq_m s_i(\al \vee \be)$ is a covering relation. Proposition \ref{prop:refinement} (i) also implies 
that $\{B(\al_i, \al), B(\al_i, \be)\} \subseteq \{-1, 0, 1\}$, and that $\al_i$ cannot be orthogonal to 
both $\al$ and $\be$.

Suppose for a contradiction that we have such an $s_i$.
We claim that in fact neither of $B(\al_i, \al)$ and $B(\al_i, \be)$ can be zero, because in the case that (say)
$B(\al_i, \be) = 0$, Proposition \ref{prop:generalrefl} (i) implies that $$
s_i(\al \vee \be) = (\al \vee \be) - B(\al_i, \al)(\al_i \vee \be)
.$$ Since $s_i(\al \vee \be)$ is assumed to be a higher root than $\al \vee \be$, we must have 
$B(\al_i, \al) = -1$. However, we also have $B(\al_i, \be) = 0$, which contradicts condition (i) of the statement.

We may now assume that $B(\al_i, \al) = \pm 1$ and that $B(\al_i, \be) = \pm 1$, with the signs chosen 
independently. By conditions (i) and (ii), the case $B(\al_i, \al) = B(\al_i, \be) = -1$ never occurs, so there 
are three other cases to consider.

The first case is $B(\al_i, \al) = -1$ and $B(\al_i, \be) = +1$. In this case, Proposition \ref{prop:generalrefl} (i)
implies that $$
s_i(\al \vee \be) 
= (\al \vee \be) + (\al_i \vee (\be - \al - \al_i))
= (\al + \al_i) \vee (\be - \al_i) 
,$$ and Proposition \ref{prop:generalrefl} (iii) implies that $\be - \al - \al_i$ is a root.
The assumption that $\rootht(\al) = \rootht(\be)$ shows that $\be - \al - \al_i$ is a negative simple root,
so that $s_i(\al \vee \be) <_2 (\al \vee \be)$, a contradiction.

The second case, where $B(\al_i, \al) = +1$ and $B(\al_i, \be) = -1$, follows by a symmetrical argument exchanging
the roles of $\al$ and $\be$.

Finally, suppose that we have $B(\al_i, \al) = B(\al_i, \be) = +1$. In this case, Proposition \ref{prop:generalrefl} (i)
implies that $$
s_i(\al \vee \be) 
= (\al \vee \be) + (\al_i \vee (\al_i - \al - \be))
= (\al - \al_i) \vee (\be - \al_i) 
.$$ Proposition \ref{prop:generalrefl} (iii) then implies that $\al_i - \al - \be$ is a root, and this root must
be negative because $\rootht(\al_i) = 1$, so that
$s_i(\al \vee \be) <_2 (\al \vee \be)$. This contradiction completes the proof.
\end{proof}

In order to state the main result about highest \tworoots, it is convenient to fix some notation for the simple reflections
in type $E_n$. (We maintain the conventions of Notation \ref{not:aandd} for types $A_n$ and $D_n$.)  We number the nodes of 
the Dynkin diagram of type $E_n$ so that $3$ is the branch node, 
$1$---$2$---$3$---$\, \cdots\, $---$(n-1)$ is a path, and the last node, $x$ is adjacent to $3$.

With these conventions, the highest root $\theta$ in type $E_6$, $E_7$, and $E_8$ is given by $$
\al_1 + 2 \al_2 + 3 \al_3 + 2 \al _4 + \al_5 + 2 \al_x
,$$ $$
2 \al_1 + 3 \al_2 + 4 \al_3 + 3 \al _4 + 2 \al_5 + \al_6 + 2 \al_x
,$$ and $$
2 \al_1 + 4 \al_2 + 6 \al_3 + 5 \al _4 + 4 \al_5 + 3 \al_6 + 2 \al_7 + 3 \al_x
,$$ respectively.

In  type $D_n$ $(n \geq 4)$, $E_6$, $E_7$ and $E_8$, there is a unique simple root, $\al_y$, that is not orthogonal to the 
highest root.  We have $\al_y = \al_2$ in type $D_n$ for all $n \geq 4$, and $\al_y = \al_x$, $\al_1$, and $\al_7$ in types
$E_6$, $E_7$, and $E_8$ respectively.

If $b$ and $c$ are nodes of the Dynkin diagram, we write $\al_{b,c}$ to mean $\sum_{i \in P} \al_i$, where $P$ is the 
set of vertices on the unique path between $b$ and $c$, counting both endpoints.

\begin{theorem}\label{thm:highestlist}
Let $W$ be a simply laced Weyl group of finite type, let $\theta$ be the highest root of $W$, and maintain the above 
notation and Notation \ref{not:aandd}. The highest \tworoot in each $W$-orbit is given as follows.
\begin{itemize}
\item[\text{\rm (i)}]{If $W$ has type $A_n$ where $n>2$, then the highest \tworoot is $$
        \al_{1,n-1} \vee \al_{2,n} = (\ep_1 - \ep_{n}) \vee (\ep_2 - \ep_{n+1})
.$$}
\item[\text{\rm (ii)}]{If $W$ has type $D_4$, then the highest \tworoots in each of the three orbits are \begin{align*}
(\theta - \al_{2,4}) \vee (\theta - \al_{2,3}) &= \al_{1,3} \vee \al_{1,4} = (\ep_1 - \ep_4) \vee (\ep_1 + \ep_4)
= (\ep_1 \vee \ep_1) - (\ep_4 \vee \ep_4), \\
(\theta - \al_{2,4}) \vee (\theta - \al_{1,2}) &= \al_{1,3} \vee \al_{3,4} = (\ep_1 - \ep_4) \vee (\ep_2 + \ep_3), \ and \  \\
(\theta - \al_{2,3}) \vee (\theta - \al_{1,2}) &= \al_{1,4} \vee \al_{3,4} = (\ep_1 + \ep_4) \vee (\ep_2 + \ep_3).
\end{align*}}
\item[\text{\rm (iii)}]{If $W$ has type $D_n$ for $n \geq 5$, then the highest \tworoot in the small orbit 
is $$\al_{1,n-1} \vee \al_{1,n} = (\ep_1 - \ep_n) \vee (\ep_1 + \ep_n) = (\ep_1 \vee \ep_1) - (\ep_n \vee \ep_n).$$}
\item[\text{\rm (iv)}]{If $W$ has type $D_n$ for $n \geq 5$, then the highest \tworoot in the large orbit is $$
(\theta - \al_{1,2}) \vee (\theta - \al_{2,3}) = (\theta - \al_1 - \al_2)  \vee (\theta - \al_2 - \al_3)
= (\ep_2 + \ep_3) \vee (\ep_1 + \ep_4)
.$$}
\item[\text{\rm (v)}]{If $W$ has type $E_6$, then the highest \tworoot is 
$(\theta - \al_{2,x}) \vee (\theta - \al_{4,x})$ =
$(\theta - \al_{2,y}) \vee (\theta - \al_{4,y}).$}
\item[\text{\rm (vi)}]{If $W$ has type $E_7$, then the highest \tworoot is 
$(\theta - \al_{x,1}) \vee (\theta - \al_{4,1})$ =
$(\theta - \al_{x,y}) \vee (\theta - \al_{4,y}).$}
\item[\text{\rm (vii)}]{If $W$ has type $E_8$, then the highest \tworoot is 
$(\theta - \al_{2,7}) \vee (\theta - \al_{x,7})$ =
$(\theta - \al_{2,y}) \vee (\theta - \al_{x,y}).$}
\end{itemize}
\end{theorem}

\begin{proof}
The proof is by Lemma \ref{lem:findhighest} in each case. Recall from Proposition \ref{prop:explicitorbits} that there
are three orbits in type $D_4$, two orbits in type $D_n$ for $n \geq 5$, and one orbit otherwise. The three \tworoots
appearing in the statement of (ii) can be distinguished by comparing components (see Definition \ref{def:components}).

In type $A_n$ where $n>2$, let $\al = \al_{1,n-1}$, and let $\be = \al_{2,n}$. The roots $\al$ and $\be$ are orthogonal roots
of height $n-1$.
The only simple root $\al_i$ for which $B(\al_i, \al) = -1$ is $\al_n$, and we have
$B(\al_n, \be) = +1$. Conversely, the only simple root for which $B(\al_i, \be) = -1$ is $\al_1$, and we
have $B(\al_i, \al) = +1$. Lemma \ref{lem:findhighest} implies that $\al \vee \be$ is the highest \tworoot in its
orbit, proving (i).

The proof of (iii) is similar to that of (i).
Suppose we are in the situation of (iii), and let $\al = \al_{1,n-1}$ and $\be = \al_{1,n}$. The roots
$\al$ and $\be$ are orthogonal roots of height $n-1$.
The only simple root $\al_i$ for which $B(\al_i, \al) = -1$ is $\al_n$, and we have
$B(\al_n, \be) = +1$. Conversely, the only simple root for which $B(\al_i, \be) = -1$ is $\al_{n-1}$, and we
have $B(\al_{n-1}, \al) = +1$. Lemma \ref{lem:findhighest} implies that $\al \vee \be$ is the highest \tworoot in its
orbit, proving (iii).

Suppose that we are in the situation of (vii), and let
$\alpha=\theta-\alpha_{2,7}$ and $\beta=\theta-\alpha_{x,y}$.  Since $\alpha_7$
is the only simple root not orthogonal to $\theta$ and $B(\theta,\alpha_7)=1$,
it follows that for any root $\gamma=\sum_{j}\lambda_j \alpha_j$ we have
$B(\theta,\gamma)=\lambda_7$, so that $B(\theta,\alpha_{2,7})=1$.  It follows
that
\[
B(\alpha,\alpha_{7})=B(\theta,\alpha_7)-B(\alpha_{2,7},\alpha_7)=1-1=0
\]
and that for every generator $s\neq 7$ we have
\[
B(\alpha,\alpha_s)=B(\theta,\alpha_s)-B(\alpha_{2,7},\alpha_s)=-B(\alpha_{2,7},\alpha_s)=
\begin{cases} 
-1 & \text{if}\; s=2,\\ 
1 & \text{if}\; s=1 \text{\ or\ } s=x,\\ 
0 & \text{otherwise.}
\end{cases}
\]
Thus, the only simple root $\alpha_i$ with $B(\alpha_i,\alpha)=-1$ is
$\alpha_2$, and we note that $B(\alpha_2,\beta)=B(\alpha_2,
\theta-\alpha_{x,7})=0-B(\alpha_2,\alpha_{x,7})=1$. Similarly, we can check
that  only simple root $\alpha_j$ with $B(\alpha_j,\beta)=-1$ is $\alpha_x$ and
$B(\alpha_x,\alpha)=1$. By Lemma \ref{lem:findhighest}, to prove $\alpha\vee\beta$ is the highest
root it now suffices to show that $\alpha$ and $\beta$ are positive real roots
of the same height. To do so, recall from \cite[Proposition 5.10 (i)]{kac90} that an
element $\gamma$ in the $\Z$-span of the simple roots is a real root if and only if
$B(\gamma,\gamma)=2$. It follows that
\[
    B(\alpha,\alpha)=B(\theta,\theta)-2B(\theta,\alpha_{2,7})+B(\alpha_{2,7},\alpha_{2,7})=2-2\cdot 1+2=2,
\]
which in turn implies that $\alpha$ is a real root. A similar argument shows
that $\beta$ is also a real root. Finally, since every simple root appears with
positive integer coefficient in $\theta$ and
$\rootht(\alpha_{2,7})=\rootht(\alpha_{x,7})$, it follows that $\alpha$ and
$\beta$ are positive roots with the same height.

The proofs of (ii), (iv), (v) and (vi) are the same as the proof of (vii), {\it mutatis mutandis}.
\end{proof}

Parts (i) and (iv) of Theorem \ref{thm:highestlist} also follow from \cite[Example 4.4]{cohen06}, which 
the authors state without proof.

We record without proof the heights of the highest \tworoot in each orbit of $W$, with respect to the canonical
basis $\B$.
The explicit decomposition of the highest \tworoot in terms of the canonical basis $\B$ is also known in each case, and we
remark that in type $E_8$, $\al_7 \vee \theta_7$ is the unique element of $\B$ that occurs with coefficient $1$
in the highest \tworoot.

\begin{center}
\begin{tabular}{ |c|c| } 
 \hline
 Orbit type & Height of highest \tworoot \\
 \hline
 $A_n$ & $(n-2)^2+1$ \\
 $D_4$, three orbits & 3, 3, 3 \\
 $D_n, n \geq 5$, small orbit & $n-1$ \\
 $D_n, n \geq 5$, large orbit & $4(n-4)(n-3)+3$ \\
 $E_6$ & $28$ \\ 
 $E_7$ & $85$ \\ 
 $E_8$ & $295$ \\
 \hline
\end{tabular}
\end{center}

The sequence for $A_n$ appears as \cite[A002522]{oeis}, and the sequence for the large orbit of $D_n$ appears as \cite[A164897]{oeis}.
The reason that the highest \tworoot in the small orbit of $D_n$ has height $n-1$ is that the highest \tworoot is the sum of all
$(n-1)$ basis elements in the small orbit, each with coefficient $1$. In turn, this is due to the phenomenon described in 
Remark \ref{rmk:smallorbit}, combined with the fact that the highest root in type $A$ is the sum of all the simple roots, each with
coefficient $1$.

\section{Irreducibility}\label{sec:repthy}

The main results of Section \ref{sec:repthy} are Theorem \ref{modradical} and Theorem \ref{s2vdecomp}, which describe the indecomposable
summands of the module $M$ in terms of \tworoots.

Recall from Corollary \ref{cor:tworootlat} (i) that the lattice of \tworoots, $\Z \B$, has the structure of
a $\Z W$-module. The following result shows that any real \tworoot is an integral linear combination of
basis \tworoots from the same $W$-orbit, and that the \tworoots from a given $W$-orbit span a submodule
of $\Z \B$.

\begin{prop}\label{intdecomp}
Let $\B$ be the canonical basis of \tworoots in type $\ypqr$, and let $\trorbit_1, \trorbit_2, \ldots, \trorbit_r$ 
be the orbits of the action of $W$ on the set of real \tworoots, $\realtworoots$. 
\begin{itemize}
    \item[\text{\rm (i)}]{The $\Z W$-module $\Z \B$ decomposes as a direct sum of $\Z W$-modules $$
\Z \B \cong \bigoplus_{i=1}^r \Z (\B \cap \trorbit_i)
.$$}
    \item[\text{\rm (ii)}]{Every \tworoot in the orbit $\trorbit_i$ lies in the submodule $\Z (\B \cap \trorbit_i)$.}
\end{itemize}
\end{prop}

\begin{proof}
Since the sets $\B \cap \trorbit_i$ partition $\B$, it follows that we have a direct sum decomposition of 
$\Z \B$ as $\Z$-modules of the form given in the statement.
Theorem \ref{thm:canbasrefl} implies that the $\Z$-modules $\Z (\B \cap \trorbit_i)$ are $\Z W$-modules,
and this completes the proof of (i).

Proposition \ref{orthorbits} (iii) shows that for any real \tworoot $\al \vee \be \in \trorbit_i$, there exists 
$w \in W$ with $$w(\al \vee \be) \in \B \cap \trorbit_i$$ for some $i$. By part (i), the result of 
applying $w^{-1}$ to $w(\al \vee \be)$ also lies in $\B \cap \trorbit_i$, which proves (ii).
\end{proof}

The Coxeter bilinear form $B$ naturally gives a $W$-invariant bilinear form on $S^2(V)$, which we 
denote by $\otherb$. It is defined as the unique linear map satisfying \begin{align*}
\otherb(\al_i \vee \al_j, \al_k \vee \al_l) 
&= \perm \left( 
\begin{matrix}
B(\al_i, \al_k) & B(\al_i, \al_l) \\
B(\al_j, \al_k) & B(\al_j, \al_l) \\
\end{matrix}
\right)\\
&= B(\al_i, \al_k)B(\al_j, \al_l) + B(\al_i, \al_l)B(\al_j, \al_k)
,\end{align*} where $\perm$ denotes the permanent of a matrix.
The form $\otherb$ restricts to an integer-valued form on the lattice $\Z \B$.
We can also make $\otherb$ into a nonzero $F$-valued form on the $FW$-module $F \B = F \otimes_\Z \Z \B$. 

\begin{lemma}\label{lem:perm}
Let $F$ be an arbitrary field and let $\otherb$ be the bilinear form on $F \B$ defined above. 
If $\al$ and $\be$ are orthogonal real roots and $v \in F \B$, then we have $$
    C_\al C_\be(v) = \otherb(\al \vee \be, v)\, (\al \vee \be)
    .$$
\end{lemma}

\begin{proof}
It is enough to consider the case where $v = v_1 \vee v_2$ is a symmetrized simple tensor, 
because the general case follows by linearity.  By Lemma \ref{lem:caction} (iii), we have $$
C_\al C_\be(v_1 \vee v_2) = (C_\al \otimes C_\be + C_\be \otimes C_\al)(v_1 \vee v_2)
.$$ Lemma \ref{lem:caction} (i) implies that $$
(C_\al \otimes C_\be)(v_1 \vee v_2) = B(\al, v_1)B(\be, v_2)(\al \otimes \be) + 
B(\al, v_2)B(\be, v_1)(\al \otimes \be)
.$$ Adding this to the analogous expression for $(C_\be \otimes C_\al)(v_1 \vee v_2)$ gives \begin{align*}
(C_\al \otimes C_\be + C_\be \otimes C_\al)(v_1 \vee v_2) &= 
\big(
B(\al, v_1)B(\be, v_2) + B(\al, v_2)B(\be, v_1)
\big) (\al \vee \be)\\ &= 
\otherb(\al \vee \be, v_1 \vee v_2) (\al \vee \be)
,\end{align*} as required.
\end{proof}

If $N$ is an $FW$-submodule of $F \B$, we define the {\it radical}, $\radi(N)$ of $N$, to be $$
\radi(N) = \{v \in N : \otherb(v, v') = 0 \text{\rm \ for\ all\ } v' \in N \} 
.$$ It is immediate from the $W$-invariance of $\otherb$ that $\radi(N)$ is a submodule of
$F \B$.

In the next result, we define the $FW$-module $F (\B \cap \trorbit)$ to be $F \otimes_\Z \Z (\B \cap \trorbit)$.

\begin{theorem}\label{modradical}
Let $\B$ be the canonical basis of \tworoots in type $\ypqr$, let $\trorbit$ be a $W$-orbit of \tworoots,
let $F$ be an arbitrary field, let $N = F(\B \cap \trorbit)$, and assume the restriction of the bilinear form
$\otherb$ to $N$ is nonzero. 
\begin{itemize}
\item[\text{\rm (i)}]{The radical $\radi(N)$ of $N$ (with respect to $\otherb$) is the unique maximal $FW$-submodule of $N$.}
    \item[\text{\rm (ii)}]{The module $N/\radi(N)$ is an irreducible $F W$-module.}
    \item[\text{\rm (iii)}]{The module $N$ is indecomposable, and is irreducible if and only if $\radi(N) = 0$.}
\end{itemize}
\end{theorem}

\begin{proof}
Since $\otherb$ is assumed not to be zero, it follows that $\radi(N)$ is a proper
submodule of $N$. To prove (i), it remains to show that every proper submodule of $N$ is contained in
$\radi(N)$. The proof reduces to showing that if $N' \leq N$ is a submodule containing an element 
$v \in N \backslash \radi(N)$ then the submodule $\langle v \rangle$ generated by $v$ is equal to 
$N$.

Fix an element $v \in N \backslash \radi(N)$.
Since $v \not\in \radi(N)$, there must be a \tworoot $\al \vee \be$ in the $W$-orbit $\trorbit$ such that 
$\otherb(\al \vee \be, v) \ne 0$. Lemma \ref{lem:perm} then implies that $\langle v \rangle$ contains 
$C_\al C_\be(v)$, which is a nonzero
multiple of $\al \vee \be$. It follows that $\langle v \rangle$ contains $\al \vee \be$, which means that
$\langle v \rangle$ contains the whole orbit $\trorbit$, and thus the whole of $N$. This completes
the proof of (i).

Part (ii) follows from part (i), and the second assertion of (iii) follows from (ii). 
If $N$ could be expressed as a nontrivial direct sum of modules $N \cong N_1 \oplus N_2$, then (i) would
imply that both $N_1$ and $N_2$ were contained in the proper submodule $\radi(N)$, which is a contradiction.
Part (iii) now follows.
\end{proof}

\begin{rmk}\label{rmk:modradical}
The requirement in Theorem \ref{modradical} that $\otherb$ should not vanish on $N$ is a mild assumption. 
This condition is always satisfied when the field does not have characteristic $2$, because any \tworoot
$\al \vee \be \in \trorbit$ satisfies $$
\otherb(\al_i \vee \al_j, \al_i \vee \al_j) = 4 \ne 0
.$$ Even in characteristic $2$, the bilinear form $\otherb$ will
not be zero provided that $W_I$ contains a parabolic subgroup of type $A_4$, because in type $A_4$ we have $$
\otherb(\al_1 \vee \al_3, \al_2 \vee \al_4) = 1
.$$ However, the form $\otherb$ is zero in type $A_3$ in characteristic $2$.
\end{rmk}

For the rest of Section \ref{sec:repthy}, we will assume that $F$ is a field of characteristic $0$.


Let $\thirdb$ be the bilinear form on $V \otimes V$ satisfying $$
\thirdb(b_i \otimes b_j, b_k \otimes b_l) = B(b_i, b_k)B(b_j, b_l)
,$$ and let $\tau : V \otimes V \rightarrow V \otimes V$ be the linear map satisfying
$\tau(b_i \otimes b_j) = b_j \otimes b_i$.
We identify the symmetric square $S^2(V)$ and the exterior square $\exterior{2}(V)$ of $V$ with the eigenspaces 
of $\tau$ for the eigenvalues $1$ and $-1$, respectively.

\begin{rmk}\label{rmk:variousb}
The forms $\thirdb$ and $\otherb$ are closely related. It follows from the definitions that the restriction of
$\thirdb$ to the module $M \leq S^2(V)$ satisfies \begin{align*}
\thirdb(\al_i \vee \al_j, \al_k \vee \al_l) 
&= 2 \big( B(\al_i, \al_k) B(\al_j, \al_l) + B(\al_i, \al_l) B(\al_j, \al_k) \big) \\
&= 2 \otherb(\al_i \vee \al_j, \al_k \vee \al_l).
\end{align*}
In particular, if the characteristic of $F$ is not $2$, the form $\otherb$ is nondegenerate on $M$ if and only if 
$\thirdb$ is.
\end{rmk}   

\begin{lemma}\label{bbnondeg}
The following are equivalent:
\begin{itemize}
    \item[\text{\rm (i)}]{$B$ is nondegenerate on $V$;}
    \item[\text{\rm (ii)}]{$\thirdb$ is nondegenerate on $V \otimes V$;}
    \item[\text{\rm (iii)}]{the restrictions of $\thirdb$ to $S^2(V)$ and to $\exterior{2}(V)$ are both 
    nondegenerate.}
\end{itemize}
\end{lemma}

\begin{proof}
To prove the equivalence of (i) and (ii), 
let $G \in M_n(k)$ be the Gram matrix of $B$ with $G_{ij} = B(b_i, b_j)$. 
The Gram matrix of $\thirdb$ is the Kronecker product $G \otimes G$, whose determinant is given 
by $(\det(G))^{2n}$. It follows that $G \otimes G$ is invertible if and only if $G$ is invertible,
and therefore that $\thirdb$ is nondegenerate if and only if $B$ is nondegenerate.

To prove the equivalence of (ii) and (iii), note that we have $$
V \otimes V \cong S^2(V) \oplus \exterior{2}(V)
$$ as $F$-vector spaces, because the characteristic of $F$ is zero. 
Setting $\al_k \wedge \al_l := \al_k \otimes \al_l - \al_l \otimes \al_k$, we have $$
\thirdb(\al_i \vee \al_j, \al_k \wedge \al_l) 
= \thirdb \big(\al_i \otimes \al_j + \al_j \otimes \al_i, \al_k \otimes \al_l - \al_l \otimes \al_k \big) 
,$$ where all the terms cancel in pairs to give zero. It follows that $S^2(V)$ and $\exterior{2}(V)$ are orthogonal
to each other with respect to the form $\thirdb$. By computing the Gram matrix $G$
of $\thirdb$ using a basis compatible with this decomposition, we obtain a block diagonal matrix whose
two blocks are the Gram matrix of $\thirdb$ restricted to $S^2(V)$ and to $\exterior{2}(V)$. It follows
that $G$ is invertible if and only if both these blocks are invertible.
\end{proof}

Now assume that the form $B$ is nondegenerate, or equivalently by Proposition \ref{virred} that 
we are not in any of the three affine types. Let $\simproots^* = \{\al^*_1, \ldots \al^*_n\}$ be the 
dual basis of $\simproots$, which we identify with a subset of $V$ in the usual way, via $$
B(\al^*_i, \al_j) = \delta_{ij}
,$$ where $\delta$ is the Kronecker delta.

Following the theory of vertex operator algebras \cite{griess98},
we define the {\it Virasoro element} of $B$ (with respect to the basis $\simproots$) to be the element $$
\virasoro = \sum_{i = 1}^n \al^*_i \otimes \al_i
.$$ In \cite[\S2]{griess98}, the element $\virasoro$ appears in the context of an algebra with identity 
${\mathbb I} = \virasoro/2$, and the bilinear form $\otherb$ is denoted by $\langle \, , \, \rangle$.
We will show that $\virasoro$ spans a complement in $S^2(V)$ to the submodule $M$. Although most of the
next result is known from the vertex operator algebras literature, we will give a self-contained proof
for the convenience of the reader and in order to fix notation.

\begin{prop}\label{prop:virasoro}
Maintain the above notation, and assume that $B$ is nondegenerate.
\begin{itemize}
\item[{\rm (i)}]{For any $v \in V \otimes V$, we have $\thirdb(\virasoro, v) = B(v)$.}
\item[{\rm (ii)}]{The Virasoro element $\virasoro$ is independent of the choice of basis $\Pi$.}
\item[{\rm (iii)}]{The Virasoro element $\virasoro$ is symmetric, meaning that $\tau(\virasoro) = 
\virasoro$, and $$\virasoro = \frac{1}{2} \sum_{i = 1}^n \al^*_i \vee \al_i
.$$}
\item[{\rm (iv)}]{The Virasoro element $\virasoro$ is fixed by the action of any $w \in W$.}
\item[{\rm (v)}]{We have $\thirdb(\virasoro, \virasoro) = n$, and $\virasoro \in S^2(V) \backslash M$.}
\end{itemize}
\end{prop}

\begin{proof}
It follows from the definitions that for all $1 \leq i, j, k \leq n$, we have $$
\thirdb(\al^*_i \otimes \al_i, \al_j \otimes \al_k) = B(\al^*_i, \al_j)B(\al_i, \al_k) = \delta_{ij} B(\al_j, \al_k) 
.$$ Part (i) follows after summing over $i$.

The form $\thirdb$ is nondegenerate by Lemma \ref{bbnondeg}, and this implies that there is a unique
element $R \in V \otimes V$ with the property that $\thirdb(R,v) = B(v)$ for all $v \in V \otimes V$. This 
implies that $\virasoro$
is characterized by the property in (i). This characterization is basis-free, proving (ii).

By part (ii), we may also define $\virasoro$ with respect to the dual basis of $\simproots$, proving that $$
\virasoro = \sum_{i=1}^n \al_i \otimes \al^*_i
.$$ Comparing this with the original definition of $\virasoro$ implies that $\virasoro = \tau(\virasoro)$, 
proving (iii).

Given $w \in W$, part (ii) shows that we can compute the Virasoro element with respect to the 
basis $\{w(\al_i)\}_{i=1}^n$. Because $B$ is $W$-invariant, the dual basis in this case is 
$\{w(\al^*_i)\}_{i=1}^n$, and we have $$
w(\virasoro) 
= \sum_{i=1}^n w(\al_i \otimes \al^*_i) 
= \sum_{i=1}^n w(\al_i) \otimes w(\al^*_i) = \virasoro
,$$ which proves (iv).

By part (iii), we have $$
\thirdb(\virasoro, \virasoro) = \thirdb\left( \sum_{i=1}^n \al^*_i \otimes \al_i, \sum_{j=1}^n \al_j \otimes \al^*_j \right)
= \sum_{i, j=1}^n B(\al^*_i, \al_j)B(\al_i, \al^*_j) = \sum_{i,j=1}^n \delta_{ij}^2 = n
,$$ which proves the first assertion of (v).
Part (i) implies that an element $v \in S^2(V)$ lies in $M$ if and only if $\thirdb(\virasoro, v) = 0$, and the second
assertion of (v) follows from (iii) and the fact that $\thirdb(\virasoro, \virasoro) \ne 0$.
\end{proof}

\begin{theorem}\label{s2vdecomp}
Let $W$ be a Weyl group of type $\ypqr$ and let $V$ be the reflection representation of $W$ over a field $F$ of 
characteristic zero. If $\ypqr$ is not of affine type, then the module $S^2(V)$ decomposes as a direct sum of
irreducible modules: the one-dimensional module $F\virasoro$, and the modules $F(\B \cap \trorbit_i)$ corresponding
to the orbits $\trorbit_i$ of \tworoots.
\end{theorem}

\begin{proof}
By Proposition \ref{virred}, the form $B$ is nondegenerate, and it follows from remarks \ref{rmk:modradical} and 
\ref{rmk:variousb} that $\otherb$ and $\thirdb$ are nonzero when restricted to each summand $F(\B \cap \trorbit_i)$.
Proposition \ref{prop:virasoro} (iv) and (v) imply that
the module $S^2(V)$ is isomorphic to $\Span(\virasoro) \oplus M$, where $\Span(\virasoro)$ affords the trivial
representation of $W$. It remains to show that the module $M$ decomposes as the direct sum of the modules
$F(\B \cap \trorbit_i)$, and that these modules are irreducible.

We first consider the case where $W$ is finite. It follows 
from Proposition \ref{intdecomp} (i), by extending scalars to $F$, that we have $$
F \B \cong \bigoplus_{i=1}^r F (\B \cap \trorbit_i)
.$$ The modules in the direct sum are indecomposable by Theorem \ref{modradical} (iii), and irreducible by Maschke's Theorem, which
completes the proof in this case.

Assume from now on that $W$ is infinite, which means by Lemma \ref{lem:componentcount} that there is a single orbit of \tworoots,
and that the module $F (\B \cap \trorbit)$ is $M$ itself.
For any subspace $N \leq S^2(V)$, define $N^\perp$ to be the subspace $$
N^\perp := \{ v \in S^2(V) : \thirdb(v, v') = 0 \text{\rm \ for\ all\ } v' \in N  \}
.$$ The nondegeneracy of $\thirdb$ on $S^2(V)$, proved in Lemma \ref{bbnondeg} (iii), shows that we always have
$\dim(N) + \dim(N^\perp) = \dim(S^2(V))$.

Assume for a contradiction that there exists a nonzero element 
$m \in M \cap M^\perp$, so that we have $M \leq \Span(m)^\perp$. The previous paragraph shows that 
$\dim(\Span(m)^\perp) = \dim(M)$, which implies that $M = \Span(m)^\perp$. However, Proposition \ref{prop:virasoro} (i) shows
that $M = \Span(\virasoro)^\perp$, and the nondegeneracy of $\thirdb$ on $S^2(V)$ then implies that $\Span(m) = \Span(\virasoro)$,
which contradicts Proposition \ref{prop:virasoro} (v). It follows that $M \cap M^\perp = \radi(M)$ is zero.
Theorem \ref{modradical} (iii) now implies that $M$ is irreducible.
\end{proof}

\begin{rmk}\label{whataboutaffine}
In the cases where $\ypqr$ is of type affine $E_n$ for $n \in \{6, 7, 8\}$, the form $B$ is degenerate and there is no obvious 
analogue of the Virasoro element. The module $M$ is indecomposable as is the case for other infinite Weyl groups, but it has an
$n$-dimensional radical $\radi(M)$ consisting of the elements $\delta \vee v$, where $v \in V$ and where $\delta$ is the lowest 
positive imaginary root. The module $\radi(M)$ is isomorphic to the reflection representation, and it in turn has a submodule
spanned by $\delta \vee \delta$.
\end{rmk}

\section{Faithfulness}\label{sec:faithful}

In Section \ref{sec:faithful}, we find the kernels of the action of the Weyl group in its action on the nontrivial summands
of $S^2(V)$. The main result is Theorem \ref{thm:faithful}, which describes when $W$ acts faithfully on the nontrivial 
summands and on the associated orbits of \tworoots. This description depends on the centre $Z(W)$ of $W$, which has
the following explicit description.

\begin{lemma}\label{lem:centre}
Let $W$ be a Weyl group of type $\ypqr$. The centre $Z(W)$ of $W$ is trivial unless $W$ is of type $E_7$, $E_8$, or $D_n$ for
$n$ even; in particular, $Z(W)$ is trivial if $W$ is infinite. In the cases where $Z(W)$ is nontrivial, we have $Z(W) = \{1, w_0\}$, 
where $w_0$ is the longest element of $W$ and $w_0$ acts as the scalar $-1$ on the reflection representation $V$.
\end{lemma}

\begin{proof}
Assume first that $W$ is finite. It follows from \cite[Exercise 6.3.1]{humphreys90} that we have $Z(W) =\{1, w_0\}$ 
in the case where the longest element $w_0$ acts as $-1$ on $V$, and $Z(W)=\{1\}$ otherwise. The assertions about $Z(W)$ being 
trivial follow from \cite[\S3.19]{humphreys90}, which completes the proof in the finite case.

Now assume that $W$ is infinite. Qi proves \cite[Proposition 2.5]{qi07}  that the center of any
irreducible
infinite Coxeter group is trivial.
In particular, this implies that the center of $W = W(\ypqr)$ is trivial if $W$ is infinite, which completes the proof.
\end{proof}

\begin{lemma}\label{lem:partition}
Let $\Gamma$ be a Dynkin diagram of type $\ypqr$ with at least five vertices, and let $\Gamma_1 \cup \Gamma_2$ be a partition of
the vertices into proper nonempty subsets. Then there are vertices $x_1 \in \Gamma_1$ and $x_2 \in \Gamma_2$ such that $x_1$ and
$x_2$ are not adjacent in $\Gamma$.
\end{lemma}

\begin{proof}
Without loss of generality, we may assume that $|\Gamma_1| \geq 3$.

If $|\Gamma_1| \geq 4$, then any $x_2 \in \Gamma_2$ fails to be adjacent to at least one element of $\Gamma_1$, because $\Gamma$ has
no vertex of degree $4$. 

If $|\Gamma_1| = 3$ then we must have $|\Gamma_2| \geq 2$. Let $x$ and $x'$ be distinct elements of $\Gamma_2$. There is a unique
vertex of degree $3$ in $\Gamma$, which means that either $x$ or $x'$ fails to be adjacent to one of the vertices in $\Gamma_1$,
completing the proof.
\end{proof}

\begin{prop}\label{prop:inffaith}
Let $W$ be an infinite group of type $\ypqr$, and let $V$ be the reflection representation of $W$. Then $W$ acts faithfully on the 
irreducible codimension-1 submodule $M$ of $S^2(V)$.
\end{prop}

\begin{proof}
Since $W$ is infinite, the rank $n$ of $W$ is at least $7$, and the module $M$ is irreducible by Theorem \ref{s2vdecomp} because there 
is a single orbit of real \tworoots. Let $w$ be a nonidentity element of $W$; we need to show that $w$ does not
act on $M$ as the identity.

Let $S_1 = \{\al_i \in \simproots : w(\al_i) < 0\}$ and let $S_2 = \simproots \backslash S_1$. The set $S_1$ is nonempty because $w \ne 1$, 
and the set $S_2$ is nonempty because otherwise, $w$ would be the longest element of $W$, which is impossible because $W$ is infinite. By 
Lemma \ref{lem:partition}, there exist orthogonal simple roots $\al_i$ and $\al_j$ such that $w(\al_i) < 0$ and $w(\al_j) > 0$. The element
$w$ sends the \tworoot $\al_i \vee \al_j$ (which is a standard basis element) to $w(\al_i) \vee w(\al_j)$, which is a negative linear 
combination of standard basis elements. In particular, $w$ does not act as the identity on $M$, which completes the proof.
\end{proof}

Note that if $W$ is a finite group and the longest element $w_0$ of $W$ acts on $V$ as the scalar $-1$, then $w_0$ will act as
the identity on $M$. In these cases, $W$ will not act faithfully on $M$. 

\begin{prop}\label{prop:finfaith}
Let $W$ be a finite Weyl group of type $D_n$ or $E_n$ with $n \geq 4$, let $V$ be the reflection representation of $W$ over
a field $F$ of characteristic zero, and let $N$ be a nontrivial irreducible direct summand of the module $S^2(V)$ corresponding to
a $W$-orbit $X$ of \tworoots.
\begin{itemize}
    \item[\text{\rm (i)}]{If $X$ is the small orbit in type $D_n$ (as in Definition \ref{def:largesmall}), or any of the three
    orbits in type $D_4$, then the kernel of 
    the action of $W$ on $N$ is elementary abelian of order $2^{n-1}$.}
    \item[\text{\rm (ii)}]{In all other cases, the kernel of the action of $W$ on $N$ is the centre, $Z(W)$.}
\end{itemize}
\end{prop}

\begin{proof}
Suppose that $X$ is one of the orbits in the statement of (i). By Remark \ref{rmk:smallorbit}, the action of $W(D_n)$ on 
$N$ factors through the action of the symmetric group $S_n$ on the root system of type $A_{n-1}$, and the latter action
is faithful. The kernel of the action is the kernel of a homomorphism from $W(D_n)$ to $S_n$ that identifies
two of the generators on the short branches. The latter is elementary abelian of order $2^{n-1}$ (see \cite[\S2.10]{humphreys90}),
which proves (i).

To prove (ii), we will show that any $w \not \in \{1, w_0\}$ acts nontrivially on $N$. 
Let $S_1 = \{\al_i \in \simproots : w(\al_i) < 0\}$ and let $S_2 = \simproots \backslash S_1$. The assumptions on $W$ mean that
$S_1$ and $S_2$ are both nonempty. By Lemma \ref{lem:partition}, there exist orthogonal simple roots $\al_i \in S_1$ and 
$\al_j \in S_2$ such that $w(\al_i) > 0$ and $w(\al_j) < 0$. 

The \tworoot $\al_i \vee \al_j$ will be in the orbit $X$ as long as $\al_i \vee \al_j$
is not in the small orbit in type $D_n$. By Proposition \ref{prop:explicitorbits} (iii), this can only happen if we are in type
$D_n$ and $\{i, j\} = \{n-1, n\}$. In this case, we can replace the pair $\{\al_{n-1}, \al_n\}$ by one of the pairs 
$\{\al_{n-3}, \al_{n-1}\}$ or $\{\al_{n-3}, \al_n\}$ to obtain a pair of orthogonal simple roots with one element from each
of $S_1$ and $S_2$.

As in the proof of Proposition \ref{prop:inffaith}, we now have a \tworoot $\al_i \vee \al_j$ that is a standard basis element 
in the orbit $X$ such that $w(\al_i \vee \al_j)$ is a negative linear combination of standard basis elements. This shows that
$w$ acts nontrivially on $N$.

If $w_0 \in Z(W)$, then $w_0$ acts as $-1$ on $V$ and $w_0$ acts trivially on $S^2(V)$ and $N$. Combined with the fact that $w$
acts nontrivially on $N$, this shows that the kernel of the action is $Z(W) = \{1, w_0\}$.

On the other hand, if $w_0 \not\in Z(W)$, then $\{1, w_0\}$ is not a normal subgroup of $W$ and $Z(W)$ is trivial. In this case, 
the kernel of the action, which we already know is contained in $\{1, w_0\}$, is also trivial, as required.
\end{proof}

\begin{rmk}\label{ref:d4faith}
If $W$ is of type $D_4$, it can be shown that the kernels of the action of $W$ on each of three nontrivial direct summands of
$S^2(V)$ intersect in the centre, $Z(W)$, of order 2.
\end{rmk}

The results of Section \ref{sec:faithful} can be summarized as follows.

\begin{theorem}\label{thm:faithful}
Let $W$ be a Weyl group of type $\ypqr$ other than $W(D_4)$, let $V$ be the reflection representation of $W$ over
a field $F$ of characteristic zero, and let $N$ be a nontrivial irreducible direct summand of the module $S^2(V)$ corresponding to
a $W$-orbit $\trorbit$ of \tworoots. If $\trorbit$ is not the small orbit in type $D_n$, then the following hold.
\begin{itemize}
\item[\text{\rm (i)}]{The kernel of the action of $W$ on $N$ is the center, $Z(W)$, of $W$.}
\item[\text{\rm (ii)}]{The group $W$ acts faithfully on $N$ if and only if one of the following conditions holds:
\begin{itemize}
\item[\text{\rm (1)}]{$W$ is infinite;}
\item[\text{\rm (2)}]{$W$ is of type $D_n$ and $n$ is odd;}
\item[\text{\rm (3)}]{$W$ is of type $E_6$.}
\end{itemize}}
\end{itemize}
\end{theorem}

\begin{proof}
Part (i) follows from Lemma \ref{lem:centre} and Proposition \ref{prop:inffaith} if $W$ is infinite, and from 
Proposition \ref{prop:finfaith} (ii) if $W$ is finite. Part (ii) follows from (i) and Lemma \ref{lem:centre}.
\end{proof}

\section*{Concluding remarks}\label{sec:conclusion}

Some natural candidates for generalizing the results of this paper including considering
$k$-roots for integers $k>2$, meaning symmetrized $k$-fold tensor products of mutually orthogonal roots. The most tractable
cases may be types $A$, $B$, and $D$, where the root systems are easy to understand and there is a diagram calculus \cite{green98}
to use as a guide. Following the completion of this paper, a notion of $k$-roots of type $D$ has been 
introduced and studied in the preprint \cite{spherical} by the first author. These $k$-roots have
similar properties to the 2-roots of the current paper and have applications to spherical functions of
Gelfand pairs $(S_n, S_k\times S_{n-k})$ arising from maximal Young subgroups of symmetric groups. 

In another direction, it would be interesting to know if the relations from
 Theorem \ref{thm:coherence2} expressing \tworoots as positive 
combinations of other \tworoots may be amenable to an interpretation in terms of categorification.


Although we did not discuss this for reasons of space, the ideas of this paper are motivated by the authors' study of the
Kazhdan--Lusztig basis of the Hecke algebra of $W$, specifically the elements $w \in W$ such that $\bfa(w)=2$, where $\bfa$ is
Lusztig's $\bfa$-function \cite{gx1,gx2}.
Using the Kazhdan--Lusztig basis $\{C_w : w \in W\}$, rather than the basis $\{C'_w : w \in W\}$, it is possible to
construct the canonical basis $\B$ as follows.
When $q$ is specialized to 1, it can be shown that for each Kazhdan--Lusztig basis element $C_w$ of $\bfa$-value 2, there are 
precisely two reflections $s_\alpha$ and $s_\beta$ such that $s_\alpha(C_w)$ and $s_\beta(C_w)$ are both equal to $-C_w$ modulo 
$I$, where $I$ is the ideal spanned by all Kazhdan--Lusztig basis elements of $\bfa$-value at least 3. It then turns out that the
function sending $C_w$ to $\al \vee \be$ extends to a module homomorphism from each cell module of $\bfa$-value 2 to the module $M$.
These identifications give rise to $q$-analogues of many of the results in this paper.

It can also be shown that for non-affine types in characteristic zero, the irreducible summands of $S^2(V)$ remain irreducible
upon restriction to the derived subgroup $W'$ of $W$. An important special case is the case $\yg{1}{2}{6}$, also known as type $E_{10}$ or
$E_8^{++}$. In this case, the derived subgroup $W'$ can be identified with $PSL(2, {\mathbf O})$ \cite[\S6.4]{feingold09}. Here, 
${\mathbf O}$ is the ring of octavians, a discrete subring of the octonions ${\mathbb O}$. 
The module $M$ in this case is a $54$-dimensional irreducible representation of $PSL(2, {\mathbf O})$ in characteristic
zero. There should be some octonionic interpretation of $M$ in this case, and we note that in this case, $M$ has twice the 
dimension of the exceptional Jordan algebra.

It follows from Remark \ref{rmk:smallorbit} that the \tworoots in the small orbit in type $D_n$ form a scaled copy of
the root lattice of type $A_{n-1}$. This suggests that the integral lattice of \tworoots $\Z \B$ may be related in interesting
ways to other known integral lattices. Some natural questions to ask are the following. 
\begin{itemize}
\item[(1)]{Are the \tworoots the only elements $x \in \Z \B$ for which $\otherb(x, x) = 4$?}
\item[(2)]{When does the lattice $\Z \B$ contain elements $x \in \Z \B$ such that $B'(x, x) = 2$?}
\item[(3)]{When does the lattice $\Z \B$ contain {\it roots}, meaning
elements $\al$ such that the reflection $$
s_\al : x \mapsto x - 2 \frac{\otherb(\al, x)}{\otherb(\al, \al)} \al
$$ gives an automorphism of $\Z \B$?}
\item[(4)]{Does the lattice $\Z \B$ have any automorphisms other than negation and those induced from automorphisms
of the root lattice $\Z \simproots$?}
\end{itemize}
The answer to question (3) above is positive in the case of the small orbit in type $D_n$. 
For question (2), Willson \cite{Willson} has shown that, outside the finite and affine types, there always exist
sign-coherent vectors $x$ with $B'(x,x)=2$. For example, consider the pairs of Coxeter diagrams of the form 
$(\Gamma', \Gamma) \in \{(Y_{2,2,2}, Y_{2,2,3}), (Y_{1,3,3}, Y_{1,3,4}), (Y_{1,2,5}, Y_{1,2,6})\}$ where $\Gamma'$ 
is naturally a subdiagram of $\Gamma$ (and where $\Gamma'$ is of affine type and $\Gamma$ is of
hyperbolic type). Let 
$\alpha_{-1}$ be the unique simple root in $\Gamma \backslash \Gamma'$, let $\alpha$ be a simple
root corresponding to one of the other endpoints of $\Gamma$, let $\delta$ be the lowest positive
imaginary root for $\Gamma'$, and let $\beta$ be a simple root of $\Gamma$ that is adjacent to
$\alpha$. Then the element $x:=((\alpha + \beta) \vee \alpha_{-1}) + (\alpha \vee \delta)$ is a sign-coherent
element with $B'(x,x)=2$. 

Finally, we note that in the physics literature, real roots in type $E_{10}$ correspond to instantons, and two real 
roots are orthogonal if and only if the corresponding instantons can ``bind at threshold"
\cite[\S 3.2]{brown04}. 
It would be interesting to know if the realization of \tworoots as lattice points in $\Z \B$ or the linear dependence 
relations between these lattice points have a physical interpretation.

\section*{Acknowledgements}
We thank Robert B. Howlett for suggesting to us a proof of Theorem \ref{thm:coherence2} on which our
current proof is based, and we thank the
referee for reading the paper carefully and suggesting many improvements. We also thank
Justin Willson for some helpful conversations.

\section*{Statements and Declarations}

The authors have no conflict of interest.

\bibliographystyle{plain}
\bibliography{2roots.bib}

\end{document}